\documentclass[headings=small,parskip=half]{scrartcl}

\usepackage[utf8]{inputenc}
\usepackage[a4paper, top=1in, bottom=1in, left=0.9in, right=0.9in]{geometry} 
\usepackage{microtype} 

\usepackage{mathtools}
\usepackage{tikz-cd}
\usepackage{xparse}
\usepackage[numbers]{natbib}
\usepackage{authblk}
\usepackage{amsmath,amssymb}
\usepackage{lipsum}
\usepackage{dsfont}
\usepackage{subcaption}
\usepackage{enumerate}
\usepackage{hyperref} 
\usepackage{multicol}
\usepackage[dvipsnames]{xcolor}
\usepackage{MnSymbol}
\usepackage{amsthm}
\usepackage{bm}
\usepackage{float}
\usepackage{graphicx}
\graphicspath{ {./} }

\newcounter{mycounter} 

\newcommand\showmycounter{\stepcounter{mycounter}\themycounter}

\newcommand{\const}{C_{\showmycounter}}

\setcitestyle{square,numbers}

\renewcommand{\P}{\mathbb{P}}

\numberwithin{equation}{section} 

\DeclareMathOperator*{\essinf}{ess\,inf} 

\DeclarePairedDelimiterX{\Norm}[1]{\|}{\|}{\Normargs{#1}}
\NewDocumentCommand{\Normargs}{>{\SplitArgument{1}{;}}m}
{\Normargsaux#1}
\NewDocumentCommand{\Normargsaux}{mm}
{\IfNoValueTF{#2}{#1} {#1\nonscript\:\delimsize\vert\allowbreak\nonscript\:\mathopen{}#2}}%
\def\norm{\Norm*}%

\DeclarePairedDelimiterX{\set}[1]{\{}{\}}{\setargs{#1}}
\NewDocumentCommand{\setargs}{>{\SplitArgument{1}{;}}m}
{\setargsaux#1}
\NewDocumentCommand{\setargsaux}{mm}
{\IfNoValueTF{#2}{#1} {#1\nonscript\:\delimsize\vert\allowbreak\nonscript\:\mathopen{}#2}}%
\def\Set{\set*}%

\providecommand{\keywords}[1]{\textbf{Keywords. } #1}

\providecommand{\MSC}[1]{\textbf{MSC (2020). } #1}

\theoremstyle{plain}
\newtheorem{theorem}{Theorem}[section] 
\newtheorem{lemma}[theorem]{Lemma}
\newtheorem{proposition}[theorem]{Proposition}
\newtheorem{corollary}[theorem]{Corollary}

\theoremstyle{definition}
\newtheorem{definition}[theorem]{Definition} 
\newtheorem{example}[theorem]{Example}
\newtheorem{remark}[theorem]{Remark}
\newtheorem{assumption}{Assumption}

\newcommand{\R}{\mathbb{R}}

\renewcommand{\d}{\,\mathrm{d}}

\date{}

\begin{document}

\title{On a stochastic phase-field model of cell motility with singular diffusion}

\author[1]{Amjad Saef}
\author[2]{Wilhelm Stannat}
\affil[1]{\small{
  Institute of Mathematics, Technische Universität Berlin,\linebreak
  Straße des 17.\ Juni 136, 10623 Berlin, Germany,\linebreak
  e-mail: \href{mailto:saef@math.tu-berlin.de}{saef@math.tu-berlin.de},\linebreak
  ORCID: \href{https://orcid.org/0009-0006-9503-3367}{0009-0006-9503-3367}}}
\affil[2]{\small{
  Institute of Mathematics, Technische Universität Berlin,\linebreak
  Straße des 17.\ Juni 136, 10623 Berlin, Germany,\linebreak
  e-mail: \href{mailto:stannat@math.tu-berlin.de}{stannat@math.tu-berlin.de},\linebreak
  ORCID: \href{https://orcid.org/0000-0002-0514-3874}{0000-0002-0514-3874}}}

\maketitle
\vspace{-1.7cm}
\begin{abstract}
We study existence of solutions in the variational sense for a class of stochastic phase-field models describing moving boundary problems. The models consist of stochastic reaction-diffusion equations with singular diffusion forced by a phase-field. We investigate both the case of an independently evolving phase-field and of coupled phase-field evolution driven by a viscous Hamilton-Jacobi equation. Such systems are used in the modelling of single-cell chemotaxis \cite{ABS}, where the contour of the cell shape corresponds to a level set of the phase-field. The technical challenge lies in the singularities at zero level sets of the phase-field. For large classes of initial data, we establish global existence of probabilistically weak solutions in $L^2$-spaces with weights which compensate for the singularities.
\end{abstract}

\MSC{35K55, 60H15, 80A22, 92C17, 92D25}

\section{Introduction}
In recent decades, the coupling of phase-field evolution with transport and reaction processes has been used to approximate solutions of free boundary problems originating from diverse subfields of physics and biology, including solidification \cite{SteinbachPhaseFieldMaterials}, tumour growth \cite{TumorPhaseField}, and cellular migration \cite{ShaoRappel}. In particular, such methods were used in modelling vesicle deformations \cite{VesiclePhaseField1, VesiclePhaseField2} and later extended to cell motility modelling \cite{ShaoRappel, ABS}. For a non-exhaustive overview of history, applications, and discussions of interpretations and implementations of phase-field methods (also called diffuse interface methods), we refer to \cite{PhaseFieldNumerics,PlappDiffuseInterface}. Their strength stems from the computational efficiency and theoretical simplicity relative to the complexity of corresponding exact formulations of moving boundary problems. We remark that mathematically rigorous investigations of how such diffuse-interface descriptions relate to sharp-interface models can be found for instance in the works of \citet{CaginalChenSharpInterface}, and in \citet{WeberInterface} for a stochastic analogue.

Phase-field models on some domain $\mathcal D$ refrain from modelling moving boundaries as sharp time-dependent hypersurfaces and instead represent the moving boundary as a diffuse transition interface.
Namely, one introduces a space-time order parameter $\phi(t, \bm x)$, $t \geq 0$, $\bm x \in \mathcal D$, which changes rapidly but continuously between two equilibria, e.g. $\phi \equiv 0$ and $\phi \equiv 1$. In cell motility models, one may then interpret the region $$\Set{\bm x \in \mathcal D:\phi(t,\bm x)\geq 0.5}$$ as the cell interior at time $t$, while the transition region represents the cell membrane or a diffuse approximation thereof. The advantage is that transport and reaction processes can be coupled to the moving geometry without explicitly parametrising the boundary.

In this article, we investigate the existence with respect to appropriate notions of solutions of stochastic phase-field models of the form \begin{equation} 
\label{IntroABSModel}
\left\{ \begin{aligned}
&\partial_t \phi(t,x) = \gamma \Delta \phi(t,x) + g(\phi, c)(t,x) + \Psi(\phi, c)(t,x)|\nabla \phi(t,x)|, \\
&\mathrm d c(t,x) = \left( D\Delta c(t,x) + D\frac{\nabla \phi(t,x)  \nabla  c(t,x)}{\phi(t,x)}  +  f(\phi,  c)(t,x) \right) \mathrm d t  +  b(\phi,c)(t,x)\mathrm d  W(t,x),
\end{aligned} \right.
\end{equation}
where  $t \geq 0$ denotes time, the domain $\mathcal D$ is the $d$-dimensional torus $\mathbb T^d$ for $d \geq 1$, and $\gamma, D > 0$ are diffusion coefficients. Moreover, $f$ and $g$ denote nonlocal reaction terms, the dispersion coefficient $b$ is a Nemytskii operator, $\Psi$ is a possibly nonlocal nonlinearity  and $W$ is coloured-in-space, white-in-time Wiener noise.

This type of model was introduced in \citet{ABS}. It was designed as a model of chemotaxis of the slime mould \textit{Dictyostelium discoideum}, which is a widely used model organism for the study of amoeboid movement and chemotaxis \cite{DictyModelOrganism}. The study of \textit{D. discoideum} is relevant beyond the specifics of this organism, since the signalling and cytoskeletal mechanisms involved in polarisation and locomotion of \textit{D. discoideum} are similar to mechanisms observed in motile eukaryotic cells in general \cite{DictyModelOrganism}.

A prominent mathematical feature of \eqref{IntroABSModel} is the singular diffusion term \begin{equation} \label{SingDiff}
   L_\phi c \coloneqq \frac{1}{\phi}\nabla \cdot \left(\phi \nabla c\right) =\Delta c +  \frac1\phi\nabla \phi \nabla c
\end{equation} which imposes an impermeable layer at the cell boundary, while allowing free diffusion inside the cell. In regions where $\phi_t$ is close to $1$, this is approximately free diffusion with diffusion coefficient $D$. Near the transition layer, however, the additional drift term $D\frac{\nabla\phi_t}{\phi_t}\nabla c_t$ compensates for outward diffusion across the diffuse membrane. More specifically, it penalises diffusion into the outwards normal direction of the transition front of the phase-field $\phi$. One obtains this diffusion term from related phase-field models of cell migration \cite{FlemmingCorticalWaves, MoldenhawerSpontaneous,CaoRappelELife, CaoRappelNr2} of the form 
\begin{equation} \label{eq:PhaseFieldUnsimplified}
\partial_t (\phi c) = \nabla \cdot (\phi \nabla c) + \phi f(\phi,c) 
\end{equation} by assuming that membrane dynamics happen on a slower time scale than those of the diffusing biochemical components, so that $\d(\phi c) \approx \phi \d c$.  Models of a similar structure as the more general equation \eqref{eq:PhaseFieldUnsimplified}
have been studied mathematically in the context of fluid dynamics for example in \cite{DanchinCriticalViscous} or \cite{VasseurNSPhaseFieldType}, see also the references therein. The authors remarks that they are not aware of any works studying stochastic extensions of the models considered in those works. 

A related class of models can be found in evolutionary ecology, where coupled reaction-diffusion systems are used to model population adaptation in heterogeneous or changing environments. The model of Pease, Lande, and Bull \cite{PeaseModel} and its variants in \cite{KirkpatrickModel,GarciaRamosModel,KanarekWebb} describe the joint dynamics of population density and mean trait distribution. These equations share structural similarities with the system considered here, in particular through logarithmic diffusion terms coupling density and trait evolution. The stochastic extension discussed in \cite{BartonSPDE} even incorporates fully noisy dynamics with singular dispersion coefficients.

Diffusions of the type \eqref{SingDiff} were first studied as so-called generalised Schrödinger operators \cite{Albeverio, RoecknerGeneralSchrödinger}. These studies were motivated by the connection of the stochastic processes generating such singular diffusions to quantum mechanics and quantum field theory \cite{NelsonDynamical}; a detailed overview of the different mathematical approaches can be found for example in \citet{UniquenessNelsonDiff}. For $\phi_t \equiv \phi$ constant, corresponding diffusion processes and (uniqueness) of semigroups generated by maximal extensions of $L_\phi$ have been studied intensively in weighted spaces $L^p(\phi)$ on both finite- and infinite-dimensional domains \cite{Wielens,Liskevich,StannatFixedTime, EberleUniquenessNonSym}. Existence and uniqueness of solutions of the associated evolution equation has been extended to the time-dependent case when the evolution of $\phi_t$ is governed by the complex Schrödinger equation \cite{CarlenDiff, StannatTimedependent}. 

At the foundation of these approaches lies the insight that the operator $L_\phi$ is symmetric only with respect to the weighted measure $\phi \d \bm x$. At a formal level, the $L^2(\phi\d \bm x)$-inner product $\langle \cdot, \cdot \rangle_{\phi}$ satisfies 
$$
\langle L_\phi c, v \rangle_\phi
=
-\langle \nabla c,\nabla v\rangle_\phi.
$$
Thus, the singular factor disappears once the equation is tested in a weighted form. This is in accordance with the phase-field interpretation, where the quantity $\phi c$ represents the amount of biochemical material distributed over the diffuse cell interior. A weighted solution concept is therefore plausible also in this approach, as it accounts for the physically meaningful quantity $\phi c$,  while the values of $c_t$ on $\phi_t \equiv 0$ remain irrelevant to the evolution of the processes.

The objective of the present work is to provide a rigorous analytical framework for time-dependent singular diffusions applied in biophysical modelling. As remarked in \cite{RaoulKirkpatrick}, this type of equation seems to have received limited attention from the mathematical community. Therein, the author studies existence and asymptotic properties of solutions of a Kirkpatrick-Barton model. Another notable work in this direction is \cite{MillerWaveKirkpatrick}.  Through the Cole-Hopf transform $z \coloneqq \log \frac{1}{\phi}$, the model is related to PDEs of the form 
\begin{equation} \label{eq:ColeHopf}
\begin{aligned}
\partial_t z_t &= \Delta z_t - |\nabla z_t|^2\\
\partial_t c_t &=  \Delta c_t - \nabla z_t\nabla  c_t 
\end{aligned}
\end{equation}
with $z \geq 0$, for $\gamma, \mu > 0$. One observes that solutions of this system are formally scaling invariant under the usual parabolic scaling. This indicates that the nonlinearities present in \eqref{eq:ColeHopf} are critical and standard parabolic estimates usually fail. Characterising well-posedness for critical nonlinearities and appropriate solution spaces is an active area of research; we note in particular the recent progress in \cite{AgrestiVeraar}. For an overview of such approaches, we refer to the recent survey articles \cite{WilkeSurvey} and \cite{VeraarSurvey}, for the deterministic and the stochastic case, respectively. Similar to our work, such approaches often work within time-weighted function spaces adapted to the criticality of the nonlinearity. Our contribution to this research is the introduction of weights which are inhomogeneous in space, where the growth near the initial time is governed by the evolution of the heat flow with suitable initial condition. This framework allows for rather general initial conditions. However, in contrast to the cited works, the question of uniqueness for unbounded initial conditions $z_0 \coloneqq \log \frac{1}{\phi_0}$ remains open. 

Our work is split into two parts. In the first part, we show existence of weighted martingale solutions of \begin{equation} 
\label{eq:WeakSolIntroDecoupled}
\d c_t = \left(\Delta c_t + \frac{1}{\phi_t}\nabla \phi_t \nabla c_t+f(\phi_t,c_t)\right)\d t + b(\phi_t,c_t) \d W_t
\end{equation} 
when $(\phi_t)_{t \in [0,T]}$ is a given process independent of $c$, and therefore call it the uncoupled 
case. In the second part we prove existence of martingale solutions of the fully coupled system 
\eqref{IntroABSModel}. 

We note that we restrict our analysis to uniformly bounded initial data and solutions by truncating the stochastic forcing outside a compact set that is left invariant under the 
nonlinear dynamics. This greatly simplifies technical considerations and respects 
biological plausibility. Thus, the singularity in \eqref{SingDiff} constitutes the main analytical obstacle in 
establishing existence of relevant notions of solutions. 

Then,
to obtain a priori estimates and apply the compactness method, a natural condition on $\frac{1}{\phi_t}\nabla\phi_t$ is
\begin{equation} \label{eq:L2condition}
\mathbb{E} \left[ \int_0^T \int \frac{|\nabla \phi_t |^2}{\phi_t^2} \d {\bm x}  
\d t \right] < \infty. 
\end{equation}
In the coupled case this turns out to be equivalent to $\log \phi_0 \in L^1$ and
$\phi_0 > 0$, $\d {\bm x}$-almost surely. 

To relax this condition, we introduce notions of weighted variational solutions in 
Definition \ref{MartSolStationary} (uncoupled case) 
and in Definition \ref{MartSolDynamic} (fully coupled case). More specifically, 
our main results, Theorems \ref{StationaryExistence} and \ref{WeakSystemExistence}, 
prove existence of martingale solutions of \eqref{IntroABSModel} w.r.t. weighted test 
functions 
$v_t = \rho^\beta_t u \in H^{1,2}(\mathbb T^n) \cap L^\infty(\mathbb T^n)$ for 
$u\in H^{1,2} (\mathbb{T}^n)\cap L^\infty(\mathbb T^n)$ and $\beta = 1$ in the uncoupled (resp. $\beta = 2$ in the coupled case). Here the weight $\rho_t$ satisfies 
\begin{equation} \label{eq:weightcondition}
\mathbb E\left[\int_0^T \int \rho^2_t \frac{|\nabla \phi_t|^2}{\phi_t^2} \d {\bm x} \d t\right] < \infty.
\end{equation}
To specify our requirements on the initial conditions, let  
$(\Omega, \mathcal F, \mathbb P)$ denote an underlying probability space. In the uncoupled 
case we need to assume that $0 \leq \phi_t \leq K_\phi$, $L_c \leq c_0 \leq K_c$ for 
all $t \in [0,T]$ almost surely, for some constants $K_\phi >0$, $L_c, K_c \in \mathbb R$  depending on the reaction terms $f$ and $g$,  
and $\nabla \phi_t \in L^2(\Omega;L^2([0,T];L^2(\mathbb T^n)))$, where 
$L^p$-spaces on the spatial domain are 
understood with respect to the Lebesgue measure. Additionally, we require that $\phi$ is absolutely 
continuous in the distributional sense with $\partial_t \phi \in L^1(\Omega;L^1([0,T];L^1(\mathbb T^n)))$. 

In the fully coupled case, we need to assume the same boundedness assumptions on the initial conditions 
$\phi_0, c_0$ and additionally, that $\nabla\phi_0 \in L^\infty(\Omega;L^2(\mathbb T^n))$. 
If in addition, $\log \phi_0 \in L^\infty(\Omega;L^1(\mathbb T^n))$, we then obtain martingale 
solutions in the classical sense. We note that by our previous comments, this condition is nearly optimal (cf. Corollary \ref{L1condition}). Even in the deterministic case, these existence 
results seem to be new.

Though methodologically similar, the two cases have different scopes. The uncoupled case in principle 
allows for functions $\phi_t$ which vanish on a set of potentially positive measure, which necessitates 
the introduction of weak derivatives in a suitable chosen weighted sense. Moreover, since there is less information at hand about the properties of $\phi$, we need to exploit the symmetrising effect of $\phi$ on the singular diffusion to control singularities at zero-level sets. This excludes the use of weights in the generality of the second part of this chapter. 

In contrast to this, 
the diffusion operator in the fully coupled equation \eqref{IntroABSModel} exhibits a singularity only 
at $t=0$: by a parabolic maximum principle, $\phi_t$ is strictly positive for $t > 0$, given any 
nonnegative, nonzero initial data, see also Prop. \ref{Subsolution}. Due to these regularising effects, we can show that for arbitrary $0 \leq \phi_0 \leq K_\phi$,
$$
\alpha \int_0^T \norm{\frac{\nabla \phi_t}{\phi^{1-\alpha}_t}}^2_{L^2} \d t < \infty
$$ 
holds uniformly in $\alpha \in (0,\frac12)$. This infinitesimal weakening of condition 
\eqref{eq:L2condition} demonstrates that for all $\alpha \in (0,\frac12)$, $\phi^\alpha$ 
is an admissible weight. 
With the compactness method applied to infinite sequences of approximating processes, 
we can derive existence of a limiting process $(\phi, c)$ such that $c$ is a weighted 
martingale solution for weights of the type $\phi^{\alpha_n}$, with $\alpha_n \to 0$. 
This constitutes a type of density result, which implies that the limiting process $c$ can be tested against any admissible 
weight. 

A commonality of both sections is the use of compactness arguments tailored to our setting, to handle the low temporal and spatial regularity of the involved terms. Existence and continuity of limit processes is then derived via additional a posteriori steps.

This work is structured as follows: In Section \ref{subsec:WeakSolSetting}, we introduce relevant notations, 
definitions and the mathematical framework of the equations we study. The core of our work begins with Section 
\ref{subsec:WeakSolStationarySect}, which defines and demonstrates existence of weighted martingale solutions of 
\eqref{eq:WeakSolIntroDecoupled}. To obtain such solutions, we first regularise the equation and solve the tamed 
equation using well-established methods. Then, we reweight accordingly to successively relax truncations using a-priori 
inequalities and compact embeddings which yield tightness of laws. In the limit, we obtain the desired type of 
solution (cf. Theorem \ref{StationaryExistence}). In Section \ref{subsec:WeakSolCoupledSection}, we tackle the 
fully coupled system \eqref{IntroABSModel}. This requires additional regularity results 
for $\phi$. Our main result, Theorem \ref{subsec:WeakSolCoupledSection}, then proves 
existence of weighted martingale solutions to \eqref{IntroABSModel} (cf. Definition \ref{MartSolDynamic}).
The last section, Section \ref{ApplicationSec}, discusses applications to equations used in the modeling 
of biophysical processes.

\section{Mathematical Setting, Notations and Main Assumptions} 
\label{subsec:WeakSolSetting}
In this section, seek to prove the existence of solutions in the so-called \textit{variational} framework. We shortly sketch the usual setting of this approach to stochastic partial differential equations. Fix a finite time $T>0$, a reflexive, separable Banach space $V$ and separable Hilbert spaces $U, H$. Note that all vector spaces in this manuscript are assumed to be real. We say that the spaces $(V,H,V^\ast)$ form a \textit{Gelfand triple} if $V$ is densely and continuously embedded into the separable Hilbert space $H$. This in turn is embedded into $V^\ast$, the dual space of $V$, by the Riesz isomorphism and the adjoint map of the continuous embedding $V \hookrightarrow H$, to obtain $V \hookrightarrow H \cong H^\ast \hookrightarrow V^\ast.$ The choice of Gelfand triple will determine the choice of test functions of variational solutions of the stochastic partial differential equation \begin{equation} \label{VarSol}
\d u_t = A(t,u_t)\d t + B(t,u_t) \d W_t.
\end{equation} Let $_{V^\ast}\langle \cdot, \cdot\rangle_V$ denote the corresponding dual pairing on $V^\ast \times V$. 

Let a filtered probability space \( (\Omega, \mathcal{F}, (\mathcal{F}_t)_{t \geq 0}, \mathbb{P}) \) satisfying the usual conditions and some (possibly nonlinear) progressively measurable operators $A \colon \Omega \times [0,T] \times V \rightarrow V^\ast$ and $B \colon \Omega \times [0,T] \times V \rightarrow HS(U,H)$ be given. Here $HS(U,H)$ is the space of Hilbert--Schmidt operators equipped with the Borel $\sigma$-Algebra induced by the corresponding norm.

\begin{definition}[Strong variational solution]
We say that there exists a \emph{probabilistically strong variational solution} of equation \eqref{VarSol} in the Gelfand triple \( V \hookrightarrow H \hookrightarrow V^\ast \) with initial distribution $\mu_{u_0}$ if for  any $\mathcal F_t$-adapted cylindrical Wiener process \( W_t \) on $U$, there exists an \( H \)-valued, \( \mathcal{F}_t \)-progressively measurable process \( u: [0,T] \times \Omega \to H \) with $u_0 \sim \mu_{u_0}$ such that \( \mathbb{P} \)-almost surely, \[ u \in L^2([0,T];V) \cap L^\infty([0,T]; H) \cap C([0,T];V^\ast)\] and the equation
\[
\langle u_t, v \rangle_H = \langle u_0, v \rangle_H + \int_0^t {_{V^\ast}}\langle A(s,u_s), v \rangle_{V} \, ds +  \left \langle \int_0^t B(s,u_s) \d W_s, v \right \rangle_H,
\]
is well-defined and holds for all \( t \in [0,T] \) and every test function \( v \in V \).
\end{definition}

\begin{definition}[Martingale solution]
\label{MarSolClassical}
We say that there exists a solution to the martingale problem associated with equation \eqref{VarSol} 
in the Gelfand triple \( V \hookrightarrow H \hookrightarrow V^\ast \) with initial distribution 
$\mu_{u_0}$ if there exists a filtered probability space \( (\Omega', \mathcal{F}',  
(\mathcal{F}'_t)_{t \geq 0}, \mathbb{P}') \) satisfying the usual conditions, random operators 
$A' \overset{d}{=} A$, $B' \overset{d}{=} B$, an $\mathcal F'$-adapted, cylindrical Wiener 
process \( W_t \) on $U$ and an \( H \)-valued, \( \mathcal{F}'_t \)-progressively measurable process 
\( u: [0,T] \times \Omega' \to H \) with $u_0 \sim \mu_{u_0}$ such that \( \mathbb{P}' \)-almost surely, 
\[ 
u \in L^2(0,T;V) \cap L^\infty([0,T]; H) \cap C([0,T];V^\ast)
\] 
and the equation
\[
\langle u_t,v \rangle_H = \langle u_0, v \rangle_H + \int_0^t {_{V^\ast}}\langle A'(s,u_s), v \rangle_{V} \, ds +  \left \langle \int_0^t B'(s,u_s) \d W_s, v \right \rangle_H,
\]
is well-defined and holds for all \( t \in [0,T] \) and every test function \(v \in V \).
\end{definition}

\begin{remark}
We note here that one can in many cases generalise the notion of a variational solution, both probabilistically weak and strong, to dual pairings $V \hookrightarrow H \cong H^\ast \hookrightarrow E$ for some space $E$ of test functions, instead of the more restrictive Gelfand triples. In this thesis, we will silently identify those two notions of solutions, provided the involved identities continue to be well-defined for the processes we consider.
\end{remark}

Throughout this section, we will use the letter $C_i,$ $i \in \mathbb N$, to denote a generic constant. All boundary problems are defined on the parabolic cylinder $[0,T] \times \mathbb T^d$ for some $T > 0$ and $d \geq 1$, where $\mathbb T^d$ denotes the flat torus. Let 
$\lambda$ be the normalised Lebesgue measure on the Borel subsets of $\mathbb T^d$. 

By $L^p(\mathbb T^d)$ and $H^{s,2}(\mathbb T^d)$, we denote the Lebesgue and Sobolev spaces on $\mathbb T^d$ for $p \geq 1$ and $s \geq 0$. Due to regularity of the domain, $H^{s,2}(\mathbb T^d)$ can be equivalently defined as either Sobolev-Slobodeckij or Bessel potential spaces. As we will deal with systems of equations, we will frequently encounter vectors of Lebesgue- or Sobolev functions. By abuse of notation, we will denote finite direct sums of Banach spaces $$\bigoplus_{i=1}^m L^p(\mathbb T^d),\qquad  \bigoplus_{i=1}^m H^{s,2}(\mathbb T^d)$$ again by $L^p(\mathbb T^d)$ resp. $H^{s,2}(\mathbb T^d)$ whenever it is clear from the context. Similarly, we will usually not highlight functions with values in such direct sums. However, we will use bold face symbols to accentuate elements of such direct sums, e.g. $\bm u \in {L}^2(\mathbb T^d)$, which should normally explain from the context the dimension of the involved identites.
Generally, we will omit reference to the spatial domain and simply write ${L}^2$ or ${H}^{s,p}$. We further adopt the notation $${H}^{1,2}_b \coloneqq {H}^{1,2} \cap {L}^\infty$$ to denote the Banach algebra of essentially bounded functions with square integrable weak derivative. 

Unless specified otherwise, the inner product $\langle \cdot, \cdot \rangle$ will always denote the inner product on (a direct sum of copies of) $L^2$. Even more generally, by $\langle \cdot, \cdot \rangle$ will denote a dual pairing of the form ${H}^{1,2} \hookrightarrow {L}^2 \cong ({L}^2)^\ast \hookrightarrow E$, where $E$ is some space of test functions. Usually, the involved spaces will become clear from the context, so we will suppress reference to them as well. 

Since we work with systems of equations, we will use the Kronecker product in the particular case 
$$
\otimes \colon \mathbb R^{1 \times d} \times \mathbb R^{k \times 1}  
\rightarrow \mathbb R^{k \times d} 
$$ 
and the operation 
$$
\odot \colon \mathbb R^{m} \times \mathbb R^{m \times n} \rightarrow \mathbb R^{m \times n}, ~ (\bm u, \bm A) \mapsto \mathrm{diag}(u_1,\dots, u_m)\cdot \bm A
$$ 
to denote the leftwise product with the diagonal matrix with entries given by $\bm u = (u_1,\dots,u_m)$. 
We further define $\nabla f \in \mathbb R^{d \times 1}$ for $f \colon \R^d \rightarrow \R$, while for $\bm g \colon \R^d \rightarrow \R^m$, we define $$\nabla \bm g \coloneqq \bm J_{\bm g} \in \R^{m \times d}.$$ With this notation at hand, we can simply write $$\nabla \cdot(f  \nabla \bm g) = f \Delta \bm g + \nabla \bm g \nabla f,$$ if we define the divergence rowwise and Laplacian componentwise, i.e. $$\Delta \bm g \coloneqq (\Delta g^1,\dots, \Delta g^m)^T\colon \mathbb T^d \rightarrow \mathbb R^m, \nabla \cdot (f \nabla \bm g) \coloneqq (\nabla \cdot (f\nabla  g^1),\dots, \nabla \cdot (f\nabla g^m))^T\colon \mathbb T^d \rightarrow \mathbb R^m.$$

We seek to establish the existence of solutions to the stochastic reaction-diffusion system  
\begin{equation} 
\label{StationaryEquation}
\d \bm c_t = \left(\frac{1}{\phi_t} \nabla \cdot(\phi_t \bm D\nabla \bm c_t) +  f(\phi_t, \bm c_t) \right) \d t +  b(\phi_t,\bm c_t)\d {\bm W}^Q_t
\end{equation} for $\bm D$ a deterministic diagonal matrix with strictly positive entries, and Wiener noise with trace class covariance $Q$. In Section \ref{subsec:WeakSolCoupledSection}, this system is moreover coupled to the parabolic equation $$\partial_t \phi_t = \gamma \Delta \phi_t + g(\phi_t,\bm c_t) +\Psi(\phi_t, \bm c_t, \nabla \phi_t).$$

\begin{assumption} \label{ass:WeakSolDispersion}We assume that the operator $ b = (b_{ij})_{1 \leq i,j\leq m}$ is a matrix of multiplication operators with entries $b_{ij}\colon \mathbb R \times \mathbb R^m \rightarrow \mathbb R$. We additionally assume that the functions $b_{ij}$ are locally Lipschitz.
\end{assumption}

\begin{assumption}\label{ass:WeakSolNemytskiiAssumption}
The nonlinearities $g,  f$ should generally correspond to Nemytskii-type operators with dependence on nonlocal properties of inputs. Due to the conditions we impose on solutions and initial conditions, it suffices to specify their behaviour on $L^\infty$. Thus, we only assume that for each $R > 0$, there exists a constant $L_R$ such that \begin{enumerate}[(i)]
    \item $g \colon L^\infty \oplus {L}^\infty \rightarrow L^\infty$ admits the bound 
    $$\begin{aligned}
    &|(g(\phi_1,\bm c_1)-g(\phi_2,\bm c_2))(\bm x)| \\&\leq L_R\left(|\phi_1(\bm x)-\phi_2(\bm x)|+ |\phi_1(\bm x)| \cdot \norm{\phi_1-\phi_2}_{L^2}+ |\phi_2(\bm x)|\left(|\bm c_1(\bm x)-\bm c_2(\bm x)|+\norm{\bm c_1-\bm c_2}_{{L}^2} \right)\right)
    \end{aligned}$$
    \item $ f \colon L^\infty \oplus {L}^\infty \rightarrow {L}^\infty$ admits the bound $$|( f(\phi_1,\bm c_1)- f(\phi_2,\bm c_2))(\bm x)|\leq L_R\left(|\phi_1(\bm x)-\phi_2(\bm x)|+|\bm c_1(\bm x)-\bm c_2(\bm x)| + \norm{\phi_1-\phi_2}_{L^2}+ \norm{\bm c_1-\bm c_2}_{{L}^2}\right)$$ 
\end{enumerate}
$\!\d {\bm x}$-almost surely whenever $\norm{\phi_1}_{L^\infty}, \norm{\phi_2}_{L^\infty}, \norm{\bm c_1}_{{L}^\infty},\norm{\bm c_2}_{{L}^\infty} \leq R$.
\end{assumption}
\begin{assumption} \label{ass:WeakSolInvariance}
Let real numbers $K_\phi$ and $$L_1, \dots L_m, K_1, \dots K_m$$ with $L_i < K_i$ be given and introduce the hypercube $\mathcal K = \prod_{i=1}^m[L_i,K_i]$ and the spaces $$\mathcal X_{\phi} = \left\{\phi \in L^\infty(\mathbb T^d) :  \phi(\bm x) \in [0,K_\phi],~ \d {\bm x}\text{-a.s.} \right\}, ~\mathcal X_{\bm c} = \left\{\bm c \in {L}^\infty(\mathbb T^d) : \bm c(\bm x) \in \mathcal K,~\d {\bm x}\text{-a.s.} \right\}.$$ We assume that  whenever $\phi \in  \mathcal{X}_\phi$ and $\bm c \in \mathcal X_{\bm c}$, then $$f_i(\phi, \bm c) \mathds{1}_{\Set{c_i \equiv L_i}}  \geq 0 \text{ and } f_i(\phi, \bm c) \mathds{1}_{\Set{c_i \equiv K_i}} \leq 0$$ and $$g(\phi, \bm c) \mathds{1}_{\Set{\phi \equiv 0}}= 0 \text{ and } g(\phi, \bm c) \mathds{1}_{\Set{\phi \equiv K_\phi}} \leq 0.$$ 
\end{assumption}
\begin{remark}
For simplicity, we will always choose $K_\phi = 1$ in this chapter.
\end{remark}
\begin{remark}
These conditions on $f$ and $g$ are modelled after nonlocal terms appearing e.g. in \ref{PhaseFieldABS}.
\end{remark}
\begin{remark}
In combination with the Lipschitz property from Assumption \ref{ass:WeakSolNemytskiiAssumption}, Assumption \ref{ass:WeakSolInvariance} implies that $\limsup_{\phi \to 0} \frac{|g(\phi,\bm c)|}{|\phi|} < \infty$, where the limit $\phi \to 0$ is taken in $L^\infty$. This property is crucial in our proof of Lemma \ref{arbitraryphibound}, and thereby crucial to the solution theory developed in this chapter.
\end{remark}

\begin{assumption} \label{ass:WeakSolPsiCoeff}
The nonlinearity $\Psi$ is of the form $$\Psi(\phi,\bm c,\nabla \phi) = \sum_{i=1}^\ell \Psi_i(\phi,\bm c) \cdot \varphi_i(\nabla \phi),$$ for some finite collection of Lipschitz functions 
$\varphi_i \colon \mathbb R^d \rightarrow \mathbb R$ with $\varphi_i(0) = 0$ and nonlinearities $\Psi_i \colon L^\infty \oplus {L}^\infty \rightarrow L^\infty$ that satisfy the Lipschitz property
$$|\left(\Psi_i(\phi_1, \bm c_1) - \Psi_i(\phi_2, \bm c_2)\right)(\bm x)| \leq L_R\left(|\phi_1(\bm x)- \phi_2(\bm x)|+|\bm c_1(\bm x)-\bm c_2(\bm x)| + \norm{\phi_1-\phi_2}_{L^2}+ \norm{\bm c_1-\bm c_2}_{{L}^2}\right)$$ whenever $\norm{\phi_1}_{L^\infty}, \norm{\phi_2}_{L^\infty}, \norm{\bm c_1}_{{L}^\infty},\norm{\bm c_2}_{{L}^\infty} \leq R$.
\end{assumption}

\begin{assumption} \label{ass:WeakSolNoiseTrunc}
As it simplifies the analysis considerably and respects the constraints of biological sensibility, we will consider strictly bounded solutions in this manuscript and modify the noise term to ensure that solution behave as such. Namely, we assume that $ b(y,\cdot) \equiv 0$ outside of $\mathcal K$, i.e. $$ b(y, \bm c) = 0, ~y \in [0,K_\phi],~\bm c \notin \mathcal K.$$
\end{assumption}
\begin{remark} \label{rem:WeakSolDispCoeffBounded}
Observe that $ b$ restricted to ${[0,K_\phi] \times \mathbb R^m}$ is globally bounded and Lipschitz by local Lipschitzness and the compact support property of Assumption \ref{ass:WeakSolNoiseTrunc}.
\end{remark}
The noise truncation allows us to crucially leverage boundedness assumptions on the initial conditions. Namely, this allows us to first truncate the nonlinearities present in the given equation 
and subsequently remove these truncations by showing that the resulting solutions remain below the 
threshold of truncation. 

\section{Solution theory for uncoupled phase-fields} 
\label{subsec:WeakSolStationarySect}
Throughout this section, we let a filtered probability space $(\Omega, \mathcal F, (\mathcal F_t)_{t\geq0}, \mathbb P)$ be given. Further, for $r > \frac d2-1 \vee 0$, let $Q\in L({L}^2)$ be symmetric and positive definite with $\sqrt{Q}\in HS({L}^2,{H}^{r,2})$. In particular, $Q$ is of trace class as an operator on ${L}^2$.

\begin{assumption} \label{ass:WeakSolStationary}
We assume that $\bm c_0 \colon \Omega \rightarrow \mathcal X_{\bm c} \subset L^\infty$ is $\mathcal{F}_0$-measurable and that 
 $$\phi \in L^2(\Omega;L^2([0,T];H^{1,2}))$$ is $\mathcal F_t$-progressively measurable. Moreover, we assume that $\phi_s \in  L^1(\Omega;W^{1,1}([0,T];L^1))$ and $$\mathbb P(\forall t  \in [0,T]\colon \,\phi_t \in \mathcal X_\phi) = \mathbb P(\bm c_0 \in \mathcal X_{\bm c}) = 1.$$ 
\end{assumption}
In particular, this means that $0 \leq \phi \leq K_\phi$. W.l.o.g., set $K_\phi = 1$. Our aim is to solve equation \eqref{StationaryEquation}under the influence of an independently moving phase-field $\phi$. Such an approach models the interplay between diffusion, reaction kinetics, and stochastic forcing when intracellular dynamics do not significantly interact with membrane dynamics. This simplification allows us to focus on the well-posedness of the stochastic partial differential equation \eqref{StationaryEquation} without the added complexity of coupled phase-field evolution. Even though this section uses similar methods as Section \ref{subsec:WeakSolCoupledSection}, it is in principle independent from the results derived therein: In this section, it is not assured that $\phi_t$ is strictly positive for $t > 0$.

Since strict positivity of $\phi \geq 0$ is not a given, we begin this section by introducing a notion of weak derivative weighted by $\phi $. This will serve as a basis for our notion of solution of \eqref{StationaryEquation}.
\begin{definition}[Weighted weak derivative] \label{def:WeightWeakDeriv}
Let $\phi \in H^{1,2}$ be nonnegative and let $$u \in L^2(\mathbb T^d)$$ be given. A weighted weak derivative $$\partial_i u \in L^2(\mathbb T^d;\phi \d {\bm x})$$ of $u$ with respect to $\phi$ is defined by the identity
$$
\int \phi \partial_i u \cdot v \d {\bm x} = - \int \partial_i \phi \cdot u v \d {\bm x} - \int \phi u \cdot \partial_i v \d {\bm x}
$$ 
for any $v \in C^\infty(\mathbb T^d)$ and $1 \leq i \leq d$.
\end{definition}
\begin{remark}
For any Radon measure $\mu$, $C^\infty(\mathbb T^d)$ is dense in $L^p(\mathbb T^d; \mu)$. Therefore, any two candidate weak derivatives $u_1, u_2$ must be equal in $L^2(\mathbb T^d; \phi \d {\bm x})$ as it would follow that $$\int \phi(u_1-u_2)v \d {\bm x} = 0$$ for all $v \in C^\infty(\mathbb T^d)$. 
\end{remark}
\begin{remark}
If $\phi, u \in L^\infty$, then existence of the weighted weak derivative implies $\phi u \in H^{1,2}(\mathbb T^d)$. In the converse direction, one can only infer that $u$ has a weighted weak derivative w.r.t. $\phi^2$.
\end{remark}
\begin{definition}[Weighted martingale solution, uncoupled case] 
\label{MartSolStationary} 
Fix $T > 0$, a trace class operator $Q \in L({L}^2)$ and a filtered probability space $(\Omega', \mathcal F', (\mathcal F'_t)_{t\geq0},\mathbb P')$. We say that $$\phi \in L^2(\Omega';L^2([0,T]; H^{1,2})\cap W^{1,1}([0,T];L^1)), ~(\phi,\bm c) \in L^\infty(\Omega' \times [0,T]; L^\infty \oplus {L}^\infty)$$ solve the martingale problem associated with equation \eqref{StationaryEquation} for noise covariance $Q$ if
\begin{enumerate}[(i)]
    \item $\phi, \bm c$ are $\mathcal F'_t$-progressively measurable such that, $\mathbb{P}'$-almost surely
    \item $\phi \bm c \in C_w([0,T];{L}^2) \cap L^2([0,T];{H}^{1,2})$ and the weak $\phi$-weighted gradient $\nabla \bm c$ exists $\d t$-almost everywhere.
    \item and, for all $\bm v \in {H}^{1,2}_b$ and $t \in [0,T]$, the identity
    $$\begin{aligned}
    \langle \bm c_t, \bm v \rangle_{\phi_t} &= \langle \bm c_0, \bm v\rangle_{\phi_0} + \int_0^t (- \langle \bm D \nabla \bm c_s, \nabla \bm v \rangle_{\phi_s} + \langle  f(\phi_s,\bm c_s),\bm v\rangle_{\phi_s} + \langle \partial_s\phi_s \bm c_s, \bm v\rangle)\d s + \left \langle \bm v, M_t \right \rangle
    \end{aligned}$$ holds, where $\bm c_0 \sim \mu_{\bm c_0}$ by abuse of notation (cf. Assumption \ref{ass:WeakSolStationary})
    \item Here, $M$ denotes a continuous, square integrable ${L}^2$-valued $\mathcal F'_t$-martingale with covariation $$\int_0^\cdot \phi_s b(\phi_s, \bm c_s) Q  b^\ast(\phi_s, \bm c_s) \phi^\ast_s\d s.$$ 
\end{enumerate}
Here, $\langle \bm u, \bm v \rangle_\phi \coloneqq \langle \phi \bm u, \bm v \rangle_{{L}^2}$ for $\bm u, \bm v \in {L}^2(\mathbb T^d; \phi \d {\bm x})$, and similarly for the respective gradients. 
\end{definition}
\begin{remark}
The type of weighted solution we introduced is suited to the regularity of $\partial_s \phi \in L^1([0,T];L^1)$. If one seeks to apply this formalism to weights with $$\partial_s \phi \in L^2([0,T];(H^{1,2})^\ast)+L^1([0,T];L^1),$$ as in Section \ref{subsec:WeakSolCoupledSection}, then one similarly needs to choose $\phi^2$ as a weight to obtain the right limits.
\end{remark}
\begin{theorem} 
\label{StationaryExistence}
Let $\phi$ be as described in the introduction to this section.  Then there exists a filtered 
probability space $(\Omega', \mathcal F', (\mathcal F'_t)_{t\geq 0}, \mathbb P')$ and a solution 
$(\phi', \bm c)$ of the weighted martingale problem (Definition \ref{MartSolStationary}) associated with 
equation \eqref{StationaryEquation} with $\phi' \sim \mu_{\phi}$ and noise covariance $Q$. In particular, it holds that $\phi \bm c \in C([0,T];{L}^2)$, with weighted weak derivative $$\nabla \bm c \in L^2(\Omega;L^2([0,T]\times \mathbb T^d;\phi_t \d \bm x \d t))$$
and $\bm c_t \in \mathcal X_{\bm c}$ for all $t \in [0,T]$, $\mathbb P'$-almost surely.
\end{theorem}
As mentioned in the discussion after Assumption \ref{ass:WeakSolNoiseTrunc}, we will first truncate the nonlinearities in the equation. To this end, let
$$ 
\bm u \mapsto \tilde{\bm u} \coloneqq (L_i \vee u_i \wedge K_i)_{1 \leq i \leq m} 
$$ 
denote the projection of any measurable function $\bm u \colon \mathbb T^d \rightarrow \mathbb R^m$ 
onto $\mathcal X_{\bm c}$ and define 
\begin{equation} \label{eq:pointwisetruncation}
\tilde{ f}(\phi, \bm u) \coloneqq  f(\phi, \tilde{\bm u})
\end{equation} 
Consider the equation
\begin{equation} \label{MoreTruncStationaryEquation}
\begin{aligned}
\d \bm c^\tau_t
&= \left(\bm D \Delta \bm c^\tau_t + \bm D\frac{\nabla \bm c^\tau_t\nabla^\tau \phi_t }{\phi_t + \epsilon} + \tilde{ f}(\phi_t, \bm c^\tau_t)\right) \d t + { b}(\phi_t,\bm c^\tau_t) \d {\bm W}^Q_t.
\end{aligned}
\end{equation}
where $\bm D$ is a diagonal matric with strictly positive entries and 
\begin{equation} \label{eq:WeakSolGradientCutoff}
\nabla^\tau u \coloneqq \begin{cases}
\nabla u & \lvert \nabla u \rvert \leq \tau \\
\tau \frac{\nabla u}{\lvert \nabla u \rvert}
&\text{else.}\end{cases}    
\end{equation}
for $\tau > 0$ and $u$ weakly differentiable. This equation is substantially simpler to solve and existence of solutions is the content of our first proposition.
\begin{proposition} \label{prop:WeakSolStationaryTruncExistence}
Assume the conditions specified at the start of this section and let ${\bm W}^Q_t$ be an $\mathcal F_t$-adapted, ${H}^{r,2}$-valued Wiener process with covariance $Q$. Then there exists a unique probabilistically strong variational solution $$\bm c^\tau \in L^2(\Omega;L^2([0,T];{H}^{1,2}) \cap C([0,T];{L}^2))$$ of \eqref{MoreTruncStationaryEquation} with $\bm c^\tau_0 =\bm c_0$.
\end{proposition}
\begin{proof}
Consider the Gelfand triple ${H}^{1,2} \hookrightarrow {L}^2 \hookrightarrow ({H}^{1,2})^\ast$. Under the specified assumptions on $\phi$, it is standard to verify the conditions of Theorem 5.1.3 in \cite{LiuRoeckner2015} for $\alpha = p = 2$ (which implies $\beta = 0$), $f_t \equiv K$ for some constant $K > 0$ and $\rho(v) \equiv 0$.
The only notable aspect is proving monotonicity in ${L}^2$ for the Hilbert--Schmidt norm of the amplitude of the noise. W.l.og. assume that $d \geq 2$. By the assumption that $r + 1 > \frac d2$, we know that there exists $\nu > d $ such that $r + \frac{d}{\nu} > \frac{d}{2}$. In the following, let $\nu' = \frac{\nu}{\nu-2}$ denote the dual exponent of $\nu/2$. We seek to show that the operator $$  f \mapsto { b}(\phi,\bm u)  f$$ is monotonous given $f \in {H}^{r,2}$. By Lipschitzness of $ b(\phi, \cdot)$ and the Hölder inequality, we find that 
$$
\begin{aligned}
\norm{({ b} (\phi_t,\bm{u_1}) - { b}(\phi_t,\bm{u_2}))  f}^2_{{L}^2} \leq \norm{{ b}(\phi_t,\bm{u_1}) - { b}(\phi_t,\bm {u_2})}^2_{{L}^{2\nu'}} \norm{  f}^2_{{L}^{\nu}}.
\end{aligned}$$ Since $\nu > 2$, the Gagliardo-Nirenberg inequality in bounded domains entails that for $\theta = d/\nu < 1$, $$\begin{aligned}
\norm{{ b}(\phi_t,\bm{u_1}) - { b}(\phi_t,\bm {u_2})}^2_{{L}^{2\nu'}} \leq \const \norm{\bm{u_1}-\bm{u_2}}^2_{{L}^{2\nu'}} 
&\leq  \const \sum_{i=1}^m \norm{\nabla ((u_1)_i-(u_2)_i)}^{2\theta}_{L^2} \norm{(u_1)_i-(u_2)_i}^{2(1-\theta)}_{L^2} \\&+ \const \norm{\bm{u_1}-\bm{u_2}}^{2}_{{L}^2}.
\end{aligned}$$ since the dissipativity of the Laplacian absorbs the excess energy after an application of Young's inequality for products, we obtain monotonicity. The Sobolev embedding ${H}^{r,2} \hookrightarrow L^\nu$ now concludes the argument.
\end{proof}

\begin{proposition} \label{prop:cboundedness}
The preceding existence proof only shows that $\bm c^\tau \in L^2(\Omega;C([0,T];{L}^2))$. However, under the assumption that $\bm c_0 \in \mathcal X_{\bm c}$ almost surely, we obtain that $\bm c^\tau_t \in \mathcal X_{\bm c}$ for all $t \in [0,T]$, $\mathbb P$-a.s.
\end{proposition}
\begin{proof}
We show that for some $\chi$ which is positive only for $\bm z \notin \mathcal K$,  $\chi(\bm c_t) \equiv 0$ whenever $\chi(\bm c_0) \equiv 0$. Hence, in particular, 
$$\bm c^\tau \in L^\infty\left(\Omega;L^\infty([0,T];{L}^\infty) \right).$$ The proof uses an infinite-dimensional Itô formula and follows the same reasoning as Theorem 2.24 in \cite{Pardoux2021} or Lemma 3.3. in \cite{Sauer2016AnalysisAA}. We exploit dissipativity of the Laplacian and the invariance Assumption \ref{ass:WeakSolInvariance} together with the truncation \ref{ass:WeakSolNoiseTrunc}. As the existence proof, it relies on boundedness of $\norm{\nabla^\tau \phi_t}_{L^\infty(\mathbb T^d)}$.    
\end{proof}
This in particular shows that this process solves \begin{equation} \label{MoreOrLessTruncStationaryEquation}
\begin{aligned}
\d \bm c^\tau_t = \left(\bm D\Delta \bm c^\tau_t +\bm D \frac{ \nabla \bm c^\tau_t\nabla^\tau \phi_t}{\phi_t + \epsilon} +  f(\phi_t, \bm c^\tau_t)\right) \d t + b(\phi_t,\bm c^\tau_t) \d {\bm W}^Q_t.
\end{aligned}
\end{equation}
In the following, let $W^{\alpha, 2}([0,T];B)$ denote the $B$-valued Sobolev-Slobodeckij space given some Banach space $B$, cf. \cite{Flandoli1995MartingaleAS}.
\begin{lemma}[\cite{SimonCompactness}] \label{p1compactness}
Let $B_0 \subset B \subset B_1$ and $\tilde B_0 \subset \tilde B$ be Banach spaces. Suppose that the injective embeddings $B_0 \hookrightarrow B$ and $\tilde B_0 \hookrightarrow \tilde B$ are compact and that $B \hookrightarrow B_1$ is continuous. Then the embeddings $$L^1([0,T];B_0) \cap W^{\alpha,1}([0,T];B_1) \hookrightarrow L^1([0,T];B)$$ and $$W^{\beta_1,q_1}([0,T];\tilde B_0) + \dots + W^{\beta_n, q_n}([0,T];\tilde B_0) \hookrightarrow C([0,T];\tilde B)$$ are compact for any $\alpha > 0$ and $\beta_1,\dots,\beta_n \in (0, 1)$, $q_1,\dots,q_n > 1,$ with $\beta_iq_i > 1$.
\end{lemma}
\begin{remark}
Note that we do not make any separability or reflexivity assumptions in the preceding theorem. However, in applications, one needs to ensure that one deals with collections of strongly measurable functions. Usually, this is implied by separability of the space $B_0$.
\end{remark}

We now show that the laws of $(\bm c^\tau)_{\tau > 0}$ are uniformly tight as measures on $L^1([0,T];{L}^2)$. Observe that for any dimension $d \geq 2$ and bounded domain $\mathcal O$, $H^{1,2}(\mathcal O)$ embeds compactly into $L^q$ for $1 \leq q < \frac{2d}{(d-2)_+}$. This holds in particular for $q = 2$. Therefore, ${H}^{1,2} \hookrightarrow {L}^2$ compactly and we can apply Lemma \ref{p1compactness}.

Then, we can extract a weakly convergent subsequence and show that its limit constitutes a solution of the martingale problem associated with the equation  
\begin{equation} \label{MoreLessTruncStationaryEquation}
\begin{aligned}
\d \bm c^\epsilon_t
&= \left(\bm D \Delta \bm c^\epsilon_t + \bm D\frac{\nabla \bm c^\epsilon_t\nabla \phi}{\phi + \epsilon}  + { f}(\phi, \bm c^\epsilon_t)\right) \d t + { b}(\phi,\bm c^\epsilon_t) \d {\bm W}^Q_t.
\end{aligned}
\end{equation} To this end, we derive uniform bounds in expectation on $(\bm c^\tau)_{\tau > 0}$. We remind the reader of our notation ${H}^{1,2}_b \coloneqq {H}^{1,2} \cap {L}^\infty$.
\begin{proposition} \label{FirstStationaryApriori}
Let ${\bm W}^Q_t$ be an $\mathcal F_t$-adapted ${H}^{r,2}$-valued $Q$-Wiener process with covariance $Q$. Then the corresponding family of solutions $(\bm c^\tau)_{\tau > 0}$ of \eqref{MoreOrLessTruncStationaryEquation} is uniformly bounded in $$L^2\left(\Omega; L^2([0,T];{H}^{1,2})\right) \cap L^1\left(\Omega;W^{\alpha,1}([0,T];({H}^{1,2}_b )^\ast)\right)$$ for arbitrary $\alpha \in (0, \frac12)$. Further, the sequence of stochastic integrals $$M^\tau = \int_0^\cdot { b}(\phi_s,\bm c^\tau_s) \d {\bm W}^Q_s$$ is uniformly bounded in $L^p(\Omega;C^{0,\alpha}([0,T];{L}^2))$ for any $\alpha \in (0,\frac12)$ and $p \in [2,\infty)$.
\end{proposition}
\begin{proof}
The proof of this bound uses essentially the same methods as the proof of Proposition \ref{stationaryApriori} and is omitted due to its simpler structure. Note that we can only obtain a bound in $W^{\alpha,1}([0,T];({H}^{1,2}_b)^\ast)$ as $\nabla^\tau \phi \nabla \bm c^\tau$ is uniformly bounded only in $L^1(\Omega; L^1([0,T];{L}^1))$.
\end{proof}

\begin{remark}
Since we only obtain that $\bm c^\tau \in W^{\alpha,1}([0,T];(H^{1,2}_b)^\ast)$, we cannot apply Lemma \ref{p1compactness} to deduce weak continuity of the limit in the compactness argument. Thus, we must justify this property a posteriori.
\end{remark}

Since the law of $\phi$ is tight on $W^{1,1}([0,T];L^1) \cap L^2([0,T];H^{1,2})$ by Ulam's tightness theorem, we can thus conclude that the laws $\mu_\tau$ of $(\phi,\partial_t \phi, \bm c_0, \bm c^\tau,M^\tau)_{\tau > 0}$ are uniformly tight as measures on the product space $$W^{1,1}([0,T];L^1) \cap L^2([0,T];H^{1,2})\times  L^1([0,T];L^1) \times {L}^2 \times L^1([0,T];{L}^2) \times C([0,T];({H}^{1,2}_b)^\ast),$$ i.e. for each $\delta > 0$ there exists a compact subset $ K_\delta$ such that $\mu_\tau(K_\delta) > 1-\delta$ for all $\tau > 0$.
By the Banach-Alaoglu theorem, we can further conclude uniform tightness of the distribution of $(\bm c^{\tau}, M^{\tau})$ in $$L^2_w([0,T];{H}^{1,2}) \times L^\infty_{w^\ast}([0,T];{L}^2),$$ where the Lebesgue spaces are equipped with the weak and weak$^\ast$ topology, respectively. By a result of \citet{JakubowskiSkorokhod}, we can thus find a probability space $(\Omega', \mathcal F', \mathbb P')$, a subsequence $(\tau_k)_{k \in \mathbb N} \to \infty$ and random variables $(\phi^{(k)},\partial_t \phi^{(k)},\bm c^{(k)}_0,\bm c^{(k)},M^{(k)})_{k \geq 1}$, $(\phi^\epsilon, \partial_t \phi^\epsilon,\bm c^{\epsilon}_0, \bm c^\epsilon,M)$ with $$(\phi^{(k)},\partial_t \phi^{(k)},\bm c^{(k)}_0,\bm c^{(k)},M^{(k)}) \overset{d}{=} (\phi, \partial_t \phi, \bm c_0, \bm c^{\tau_k},M^{\tau_k})$$ and $$(\phi^{(k)},\partial_t \phi^{(k)},\bm c^{(k)}_0,\bm c^{(k)},M^{(k)}) \overset{\mathcal A}{\to} (\phi^\epsilon, \partial_t \phi^\epsilon, \bm c^{\epsilon}_0, \bm c^\epsilon, M),$$ $\mathbb P'$-almost surely, for \begin{equation} \label{CompactnessSpace}
\begin{aligned}
\mathcal A & \coloneqq W^{1,1}([0,T];L^1) \cap L^2([0,T];H^{1,2})\times  L^1([0,T];L^1) \times {L}^2 \\ &\times L^1([0,T];{L}^2) \cap L^2_w([0,T];{H}^{1,2}) \times C([0,T];\bm ({H}^{1,2}_b)^\ast) \cap L^\infty_{w^\ast}([0,T];{L}^2).
\end{aligned}
\end{equation}
\begin{proposition} \label{StationaryApproximation}
The process $\bm c^\epsilon \in L^2(\Omega';L^2([0,T]; H^{1,2}) \cap C([0,T];{L}^2))$ solves the martingale problem associated with equation \eqref{MoreLessTruncStationaryEquation}, with $\bm c_0 \sim \mu_{\bm c_0}$, $\phi^\epsilon \sim \mu_\phi$ and noise covariance $Q$. Additionally, $\bm c^\epsilon_t \in \mathcal X_{\bm c}$ for all $t \in [0,T]$, almost surely.
\end{proposition}
\begin{proof} First note that almost sure convergence of $\bm c^{(k)}$ in $L^1([0,T];{L}^2)$ also implies the convergence of $$(\bm c^{(k)})_k \subset L^1(\Omega'; L^1([0,T];{L}^1))$$by the Lebesgue DCT. The reason is the $\Omega \times [0,t]$-uniform boundedness $\bm c^{(k)} \in \mathcal X_{\bm c}$ for all $t$, almost surely. As a consequence, we can choose a subsequence $(k_m)_{m \geq 1}$ such that $\bm c^{(k_m)} \to \bm c^\epsilon$ almost surely on $\Omega' \times [0,T]\times \mathbb T^d$. Similarly, the sequences $(\phi^{(k)})_{k \geq 1}$, $(\nabla \phi^{(k)})_{k \geq 1}$ are identically distributed and therefore uniformly integrable, so we can apply the same reasoning and find $L^2(\Omega' \times [0,T]\times  \mathbb T^d)$- and $\d \mathbb P'\otimes \d t \otimes \d {\bm x}$-a.s. convergent subsequences. 
To simplify notation, denote this subsequence again by $k \geq 1$.  

These convergences in particular imply that $\bm c^\epsilon_t \in \mathcal X_{\bm c}$ on a set of $\d \mathbb P' \otimes \d t$ full measure and that $\phi^\epsilon \sim \mu_\phi$, since $\mathbb E \left[f(\phi^\epsilon)\right] = \mathbb E\left[f(\phi^{(1)})\right]$ for all continuous, bounded functions with domain $W^{1,1}([0,T];L^1) \cap L^2([0,T];H^{1,2})$. Further, $\partial_t \phi^\epsilon = \lim_{n \to \infty} \partial_s \phi^{(k)}$ and $\bm c^\epsilon_0 \sim \mu_{\bm c_0}$ follows by similar arguments.

By virtue of the convergences we obtained, we arrive at the relationships
$$
\begin{tikzcd}[column sep = tiny]
\langle \bm c^{(k)}_t, \bm v \rangle \arrow[d, "\mathbb P' \otimes \d t -\text{a.s.}"] \arrow[r,equals]& \langle \bm c^{(k)}_0, \bm v\rangle
+ \int_0^t -\langle \bm D \nabla \bm c^{(k)}_s, \nabla \bm v \rangle
+ \left \langle \bm D\frac{ \nabla \bm c^{(k)}_s\nabla \phi^{(k)}_s}{\phi^{(k)}_s+\epsilon}, \bm v \right \rangle
+ \langle  f(\phi^{(k)}_s, \bm c^{(k)}_s), \bm v \rangle \, \mathrm ds
+ \langle M^{(k)}_t, \bm v\rangle \arrow[d, "\mathbb P'-\text{a.s.}"] \\
\langle \bm c^{\epsilon}_t, \bm v \rangle  \arrow[r,equals,shorten <= 1.3em, shorten >= 1.3em,"\mathbb P' \otimes \d t -\text{a.s.}"] & \langle \bm c^{\epsilon}_0, \bm v\rangle
+ \int_0^t- \langle \bm D\nabla \bm c^{\epsilon}_s, \nabla \bm v \rangle
+ \left \langle \bm D\frac{ \nabla \bm c^{\epsilon}_s\nabla \phi^{\epsilon}_s}{\phi^{\epsilon}_s+\epsilon}, \bm v \right \rangle
+ \langle  f(\phi^{\epsilon}_s, \bm c^{\epsilon}_s), \bm v \rangle \, \mathrm ds
+ \langle M_t, \bm v\rangle
\end{tikzcd}
$$
Convergence of the second term inside the integral in particular is a result of the weak convergence $\nabla \bm c^{(k)} \to \nabla \bm c^\epsilon$ and the strong convergence $$\norm{\left(\frac{\nabla^{\tau_k} \phi^{(k)}}{\phi^{(k)}+\epsilon}-\frac{\nabla \phi^{\epsilon}}{\phi^{\epsilon}+\epsilon}\right)\otimes \bm v}_{ L^2([0,T];{L}^2)} \leq \norm{\bm v}_{{L}^\infty} \norm{\left(\frac{\nabla^{\tau_k} \phi^{(k)}}{\phi^{(k)}+\epsilon}-\frac{\nabla \phi^{\epsilon}}{\phi^{\epsilon}+\epsilon}\right)}_{ L^2([0,T];{L}^2)}\to 0,$$ on a suitably chosen subsequence which does not depend on $\omega \in \Omega'$. The existence of such a subsequence follows from $L^2(\Omega'\times [0,T] \times \mathbb T^d)$-convergence of $\frac{\nabla^{\tau_k} \phi^{(k)}}{\phi^{(k)}+\epsilon}$, which is a consequence of $\d \mathbb P' \otimes\d t\otimes \d {\bm x}$-almost sure convergence and uniform integrability of identically distributed random variables.

In other words, we see that \begin{enumerate}
    \item[(i)] For almost every $\omega \in \Omega'$, there exists a $\!\d t$-null set $\mathcal N_\omega \subset [0,T]$ such that \begin{equation} \label{StationaryApproxVariationalIdentity}
    \langle \bm c^\epsilon_t, \bm v \rangle = \langle \bm c^\epsilon_0, \bm v\rangle + \int_0^t -\langle \bm D\nabla \bm c^\epsilon_s, \nabla \bm v \rangle + \left  \langle \bm D\frac{ \nabla \bm c^\epsilon_s\nabla \phi^\epsilon_s}{\phi^\epsilon_s+\epsilon}, \bm v \right \rangle + \langle  f(\phi^\epsilon_s, \bm c^\epsilon_s), \bm v \rangle \d s + \langle M_t, \bm v\rangle,
    \end{equation} for all $\bm v \in {H}^{1,2}_b$ and $t \notin  \mathcal N_\omega$. 
    \item[(ii)] $\mathbb P'$-almost surely, $\langle \bm c^{(k)}_t, \bm v \rangle$ converges to the right hand side of \eqref{StationaryApproxVariationalIdentity} for all $t\in [0,T]$ and $ v \in {H}^{1,2}_b$.
\end{enumerate}
Since ${H}^{1,2}_b \subset {L}^2$ is dense and $\bm c^{(k)}_t \in \mathcal X_{\bm c} \subset {L}^\infty$, conclusion (ii) can be strengthened to weak convergence in ${L}^2$ for all $t \in [0,T]$, $\mathbb P'$-almost surely. Denote this limit by $\tilde{\bm c}_t$. Since $\mathcal X_{\bm c}$ is convex and closed under strong convergence, it is closed under weak convergence, and we find that $\tilde{\bm c}_t \in \mathcal X_{\bm c}$ for all $t \in [0,T]$, almost surely. Combined with continuity of the right hand side of \eqref{StationaryApproxVariationalIdentity}, uniform boundedness of $\tilde{\bm c}_t$ additionally implies that $\tilde{\bm c}_t$ is weakly continuous. After demonstrating the martingale property of $M_t$, application of Lemma \ref{MySquareIto} actually yields strong continuity of $\bm c^\epsilon \in {L}^2$, since $\bm c$ is weakly continuous and $t \mapsto \norm{\bm c^\epsilon_t}_{{L}^2}$ is continuous.

Since, on the other hand, $\bm c^{(k)}_t$ converges strongly to $\bm c^\epsilon_t$ for $t \notin \mathcal N_\omega$, we find that $\tilde{\bm c} \equiv \bm c^\epsilon$ in $L^1([0,T];{L}^2)$. From hereon, identify $\tilde{\bm c}$ with $\bm c^\epsilon$. Altogether, this then shows that $\bm c^\epsilon$ satisfies equation \eqref{StationaryApproxVariationalIdentity} for all $t\in [0,T]$.

We now demonstrate the claimed properties of $M$.
\begin{enumerate}
    \item[(Adaptedness)] 
    By rearranging \eqref{StationaryApproxVariationalIdentity}, we can identify $\langle M_t, \bm v\rangle$ as a measurable function of $\bm c^\epsilon_0$ and $$ \phi^\epsilon|_{[0,t]},\bm c^\epsilon|_{[0,t]}, \bm c^\epsilon_t \in X_t \coloneqq L^2([0,t];H^{1,2}) \times L^2_w([0,t];{H}^{1,2}) \cap L^2([0,t];{L}^2) \times  {L}^2_w$$ for arbitrary $v \in {H}^{1,2}_b$, $t > 0$. Thus $\langle M_t,\bm v \rangle$ is adapted to $\mathcal F_t = \sigma(\phi^\epsilon|_{[0,s]},\bm c^\epsilon|_{[0,s]}, \bm c^\epsilon_s\,; s \leq t)$.
    \item[(Martingale)] 
    First, note that \begin{equation}\label{MartingaleBoundedness}
    M \in L^\infty([0,T];{L}^2) \cap C([0,T];({H}^{1,2})^\ast)
    \end{equation} implies that $M$ has a weakly continuous representant $\widetilde M_t$ with values in ${L}^2$. Now, density of ${H}^{1,2}_b \subset {L}^2$ implies adaptedness  of $\widetilde M_t \in {L}^2$ to $\mathcal F_t$.
    
    To prove that $\widetilde M$ is not only adapted, but a martingale, we aim to show that $$\begin{aligned}
    &\mathbb E' \left[\langle \widetilde M_t-\widetilde M_s, \bm v \rangle \psi(\phi^\epsilon|_{[0,s]},\bm c^\epsilon|_{[0,s]}, \bm c^\epsilon_s) \right] \\&= \lim_{k \to \infty} \mathbb E' \left[\langle M^{(k)}_t-M^{(k)}_s, \bm v \rangle \psi(\phi^{(k)}|_{[0,s]}, \bm c^{(k)}|_{[0,s]},\bm c^{(k)}_s) \right]  = 0
    \end{aligned}$$ for arbitrary $t>s \geq 0$, $\bm v \in {L}^2$ and bounded continuous function   $\psi \colon  X_s\rightarrow \mathbb R$. This case extends directly to the case of a bounded cylindrical function $\psi$ and thereby proves that the conditional expectation with respect to $\mathcal F_s$ vanishes.
    
    Observe that $(\langle M^{(k)}, \bm v\rangle)_{k \geq 1}$ is a sequence of real valued martingales with quadratic variation 
    $$
    \int_0^T \norm{\sqrt{Q} b^\ast(\phi^{(k)}_s, \bm c^{(k)}_s)\bm v}^2_{{H}^{r,2}} \leq \const T \norm{\sqrt Q}^2_{HS} \norm{v}^2_{{L}^2}
    $$ 
    uniformly bounded in $k$, by $\d \mathbb P' \otimes \d t \otimes \d t$-boundedness of $\phi$ and $\bm c$. 
    By the BDG inequality, $L^p
    $-boundedness of the sequence follows for any $p > 1$. Thus, if we can show almost sure convergence for $\bm v \in {L}^2$, we obtain the martingale property. But the fact that 
    $$\langle M^{(k)},\bm v \rangle \overset{C([0,T])}{\to} \langle \widetilde M,\bm v \rangle, ~ \mathbb P'\text{-a.s.}
    $$ already follows as an implication of uniform convergence of $(M^{(k)})_{k \geq 1} \subset C([0,T];\bm ({H}^{1,2})^\ast)$ together with uniform boundedness of this sequence in $C([0,T];{L}^2)$. 
    We can conclude that this convergence holds in 
    $L^p(\Omega';C([0,T]))$ for any $p > 1$ and thus demonstrate the claim.
    \item[(Covariance)] Analogously (cf. \cite{Flandoli1995MartingaleAS}), the quadratic variation of the continuous martingale $\langle \widetilde M,\bm v \rangle$ can be identified by dominated convergence, where we utilise $\d \mathbb P' \otimes \d t \otimes \d {\bm x}$-almost sure convergence of $\bm c^{(k)}$ and $\phi^{(k)}$.
    \item[(Continuity)] 
    We need to ensure that $\widetilde M$ is a continuous, ${L}^2$-valued martingale. If we can show that for some basis $(\bm e_k)_{k \geq 1}$ of ${L}^2$,
    \begin{equation} \label{MartingaleInftyBound}
        \sum_{k \geq 1} \sup_{0 \leq t \leq T} \langle \widetilde M_t, \bm e_k\rangle^2 < \infty,
        \end{equation} then continuity of $t \mapsto \norm{\widetilde M_t}_{L^2}$ follows by the Lebesgue DCT and thereby, $\widetilde M$ is continuous in ${L}^2$. It is left to show that \eqref{MartingaleInftyBound} holds. But by the BDG inequality, we find that $$\begin{aligned}
        \mathbb E'\left[\sum_{k \geq 1} \sup_{0 \leq t \leq T} \langle \widetilde M_t, \bm e_k\rangle^2\right] = \sum_{k \geq 1} \mathbb E' \left[\sup_{0 \leq t \leq T} \langle \widetilde M_t, \bm e_k\rangle^2\right] &\leq \const \sum_{k \geq 1} \mathbb E'\left[\int_0^T \norm{\sqrt{Q} b^\ast(\phi^\epsilon_s, \bm c^\epsilon_s)\bm e_k}^2_{{L}^2} \d s\right] \\ &= C_{\themycounter}\,\mathbb E' \left[ \int_0^T \norm{\sqrt{Q} b^\ast(\phi^\epsilon_s, \bm c^\epsilon_s)}^2_{HS({L}^2)}\d s\right] < \infty.
        \end{aligned}$$
\end{enumerate}
Altogether, we can conclude that $\widetilde M$ is a square-integrable, continuous ${L}^2$-valued martingale with quadratic variation process $\int_0^\cdot b(\phi^\epsilon_s, \bm c^\epsilon_s) Qb^\ast(\phi^\epsilon_s, \bm c^\epsilon_s)\d s$ such that $\langle \widetilde M, \bm v\rangle = \langle M, \bm v \rangle$ for all $\bm v \in {H}^{1,2}_b$.
\end{proof}

We now repeat the previous compactness argument with the sequence of weighted processes $$((\phi^\epsilon+\epsilon)c^\epsilon)_{\epsilon > 0}.$$ Due to the low integrability of $\nabla \phi^\epsilon \nabla c^\epsilon$, these processes do not fulfill the conditions of standard identities, which is why we reprove them in the specific setting we encounter here. This is the content of the next two lemmata.
\begin{lemma} \label{MyProdIto}
Let a filtered probability space $(\Omega, \mathcal F, (\mathcal F_t)_{t \geq 0}, \mathbb P)$ and adapted processes $$x_s, y_s \in L^2(\Omega;L^2([0,T];H^{1,2})) \cap L^\infty(\Omega;L^\infty([0,T];L^\infty))$$ be given. Moreover, assume that there exist $$u_s, v_s \in L^2(\Omega;L^2([0,T];\bm (H^{1,2})^\ast)), ~\tilde u_s, \tilde v_s \in L^1(\Omega;L^1([0,T];L^1)),$$ such that almost surely, $$\langle x_t, w \rangle = \langle x_0, w \rangle + \int_0^t \langle u_s+ \tilde u_s, w\rangle \d s + \langle w, M_t \rangle$$ and 
$$\partial_s \langle y_s,w\rangle = \langle v_s+ \tilde v_s, w\rangle$$ for any $w \in H^{1,2}_b$. Here $M$ is assumed to be a continuous square-integrable ${L}^2$-valued martingale with respect to $(\mathcal F_t)_{t \geq 0}$ whose covariation is given by $\int_0^t g_s Qg^\ast_s\d s$, for some $g \in L^\infty([0,T];L_2(U,L^2))$. Then these processes are weakly continuous in $L^2$, and their product $$x_t y_t \in L^2(\Omega;L^2([0,T];H^{1,2})) \cap L^\infty(\Omega;L^\infty([0,T];L^\infty))$$ is weakly continuous such that $$\begin{aligned}
\langle x_ty_t,w\rangle = \langle x_0y_0,w\rangle + \int_0^t \Big(\langle u_s+\tilde u_s,y_s w\rangle +\langle v_s+ \tilde v_s,x_s w\rangle\Big)\d s + \langle w, \int_0^t y_s \d M_s \rangle,
\end{aligned}$$
for any $w \in H^{1,2}_b$ and $t \in [0,T]$, $\P$-almost surely.
\end{lemma}

\begin{lemma} \label{MySquareIto}
Consider the setting of the previous lemma. It then holds that \begin{equation}
\begin{aligned}
\norm{x_t}^2_{L^2} &= \norm{x_0}^2_{L^2} + \int_0^t 2\langle u_s + \tilde u_s, x_s \rangle \d s + \int_0^t 2\langle x_s, \cdot \rangle \d M_s +\int_0^t \norm{g_s \circ \sqrt Q}^2_{HS} \d s,
\end{aligned}
\end{equation}
for any $t \in [0,T]$, $\P$-almost surely.
\end{lemma}
\begin{proposition} \label{stationaryApriori}
To improve readability, let $(\Omega',\mathcal F', \mathbb P')$ denote a probability space containing a family of solutions $(\bm c^{\epsilon})_{\epsilon > 0}$ of \eqref{MoreLessTruncStationaryEquation}. Then there exists a uniform bound on $$((\phi^\epsilon+\epsilon)\bm c^{\epsilon})_{\epsilon > 0} \subset L^2\left(\Omega'; L^2([0,T];{H}^{1,2})\right) \cap L^1\left(\Omega';W^{\alpha, 1}([0,T];({H}^{1,2}_b)^\ast)\right)$$ for arbitrary $\alpha \in (0, 1/2)$, and in particular, $$\sup_{\epsilon > 0}\mathbb E'\left[\int_0^T \norm{\sqrt{\phi^\epsilon_s+\epsilon} \bm D^{\frac 12}\nabla \bm c^\epsilon_s}^2_{{L}^2}\d t \right] < \infty.$$ Further, the sequence of unbounded variation parts $(\int_0^t (\phi^\epsilon_s+\epsilon) \d M^\epsilon_s)_{\epsilon > 0}$ of $(\phi^\epsilon+\epsilon)\bm c^\epsilon$ is uniformly bounded in $L^p(\Omega';C^{0,\alpha}([0,T];{L}^2))$ for any $\alpha \in (0,\frac12)$ and $p \in [2,\infty)$.
\end{proposition}
\begin{proof}
We first demonstrate the second claim. Per Proposition \ref{StationaryApproximation}, we know that the unbounded variation part of $(\phi^\epsilon+\epsilon)\bm c^\epsilon$, as an $({H}^{1,2}_b)^\ast$-valued process, is given by $\langle \cdot,\int_0^t(\phi^\epsilon+\epsilon)\d M^\epsilon_s\rangle$, for some continuous, square-integrable ${L}^2$-valued martingale $M^\epsilon$ with covariation $$\int_0^t  b(\phi^\epsilon_s, \bm c^\epsilon_s)Q b^\ast(\phi^\epsilon_s,\bm c^\epsilon_s)\d s.$$  We will prove existence of a uniform bounded in expectation of the $W^{\alpha, p}([0,T];{L}^2)$-norm for any $\alpha \in (0,\frac12)$ and $p \geq 2$, so that the embedding $$W^{\alpha, p}([0,T];{L}^2) \hookrightarrow C^{0,\alpha- \frac1p}([0,T];{L}^2), ~ p > \frac1\alpha$$ implies the claim. To this end, note that by an application of the martingale representation theorem and Lemma 2.1 in \cite{Flandoli1995MartingaleAS}, we know that for any $\alpha \in (0, \frac12)$ and $p \geq 2$, there exists a constant $C(\alpha, p)$ such that 
$$\mathbb E'\left[\norm{\int_0^\cdot h(s)  \d M^\epsilon_s}^p_{W^{\alpha, p}([0,T];{L}^2)}\right] \leq C(\alpha, p) \mathbb E '\left[\int_0^T \norm{ h(t)    b(\phi^\epsilon_t, \bm c^\epsilon_t)\sqrt Q}^p_{HS} \d t\right]
$$ for any progressively measurable process $h$ with values in the space of linear operators on ${L}^2$ \cite{Flandoli1995MartingaleAS}. It follows that 
$$\begin{aligned}
   \mathbb E'\left[\norm{\int_0^\cdot  (\phi^\epsilon_s+\epsilon)\d M^\epsilon_s}^p_{W^{\alpha,p}([0,T];{L}^2)}\right] \leq C(\alpha,p) \mathbb E'\left[\int_0^T \norm{(\phi^\epsilon_t+\epsilon) b(\phi^\epsilon_t, \bm c^\epsilon_t) \circ \sqrt{Q}}^p_{HS} \d t \right]
\end{aligned}$$ is bounded above uniformly in $\epsilon > 0$ by uniform boundedness of $\phi^\epsilon, \bm c^\epsilon\in L^\infty$ (cf. Remark \ref{rem:WeakSolDispCoeffBounded}). This demonstrates the second claim.

We now prove the first claim. By the previous lemmata, we can apply an Itô-formula to the functional $u \mapsto \int (\phi^\epsilon_s+\epsilon) |u|^2 \d {\bm x}$ and find that 
$$
\begin{aligned}
\norm{\sqrt{\phi^\epsilon_t+\epsilon} \bm c^\epsilon_t}^2_{{L}^2} = & \norm{\sqrt{\phi^\epsilon_0+\epsilon} \bm c^\epsilon_0}^2_{{L}^2}  + \int_0^t \Bigg( -2 \norm{\sqrt{\phi^\epsilon_s+\epsilon} \bm D^{\frac 12}\nabla \bm c^\epsilon_s}^2_{{L}^2} + 2\langle(\phi^\epsilon_s+\epsilon)  f(\phi^\epsilon_s, \bm c^\epsilon_s), \bm c^\epsilon_s \rangle
\\ &+\norm{ \sqrt{\phi^\epsilon_s+\epsilon}  b(\phi^\epsilon_s, \bm c^\epsilon_s) \circ \sqrt Q}^2_{HS} + \langle \partial \phi^\epsilon_s \bm c^\epsilon_s,\bm c^\epsilon_s\rangle\Bigg)\d s + 2\int_0^t \langle (\phi^\epsilon_s+\epsilon) \bm c^\epsilon_s,\d M^\epsilon_t \rangle .
\end{aligned}
$$
Note here in particular the cancellation due to the structure of the logarithmic correction term. By rearranging and using boundedness of $\bm c^\epsilon \in \mathcal X_{\bm c}$, $\phi^\epsilon \in \mathcal X_\phi$ and $\partial_s\phi^\epsilon \in L^1(\Omega;L^1([0,T];L^1))$, we almost surely find that
$$\begin{aligned}
&\int_0^T \norm{\sqrt{\phi^\epsilon_s+\epsilon} \bm D^{\frac 12}\nabla \bm c^\epsilon_t}^2_{{L}^2} \d t \leq \const (1+\mathrm{tr}\, Q) T + \int_0^T \langle (\phi^\epsilon_t+\epsilon) \bm c^\epsilon_t,  \d M^\epsilon_t \rangle.
\end{aligned}$$ Since the stochastic integral is almost surely finite and in particular centered, there exists some $\const > 0$ independent of $\epsilon > 0$ such that
$$\mathbb E'\left[\int_0^T \norm{\sqrt{\phi^\epsilon_s+\epsilon} \bm D^{\frac 12}\nabla \bm c^\epsilon_s}^2_{{L}^2}\d t \right] \leq C_{\themycounter}.$$ We now show boundedness in expectation of $(\phi^\epsilon+\epsilon)\bm c^\epsilon$ in the spaces $W^{\alpha,1}([0,T];({H}^{1,2}_b)^\ast)$. To this end, observe that for $\bm v \in {H}^{1,2}_b$, $$\begin{aligned}
\langle (\phi^\epsilon_t+\epsilon) \bm c^\epsilon_t, \bm v\rangle &= \langle (\phi^\epsilon_0+\epsilon) \bm c^\epsilon_0, \bm v\rangle + \int_0^t -\langle (\phi^\epsilon_s+\epsilon)\bm D\nabla \bm c^\epsilon_s, \nabla \bm v \rangle + \langle (\phi^\epsilon_s+\epsilon) f(\phi^\epsilon_s,\bm c^\epsilon_s),\bm v\rangle \d s \\ &\quad + \int_0^t \langle \partial_s\phi^\epsilon_s \bm c^\epsilon_s, \bm v \rangle \d s +\left \langle \bm v, \int_0^t \bm (\phi^\epsilon_s+\epsilon) \d M^\epsilon_s \right \rangle \\
&\coloneqq I^\epsilon_t(\bm v) + J^\epsilon_t(\bm v).
\end{aligned}$$
Then the preceding bound on $(\phi^\epsilon+\epsilon)\nabla \bm c^\epsilon_t$ and almost sure boundedness of $\phi^\epsilon$ and $\bm c^\epsilon$ yield that $$\begin{aligned}
   \sup_{\epsilon > 0} \mathbb E'\left[\norm{I^\epsilon}^2_{W^{\alpha,1}([0,T];({H}^{1,2}_b)^\ast)}\right] \leq 
   \sup_{\epsilon > 0}E'\left[\norm{I^\epsilon}^2_{W^{1,1}([0,T];({H}^{1,2}_b)^\ast)}\right] < \infty.
\end{aligned}$$
The previously derived estimate on $M^\epsilon$ concludes the proof.
\end{proof}

\begin{proof}[Proof of Theorem \ref{StationaryExistence}]
Again, we find a probability space $(\Omega', \mathcal F', \mathbb P')$, a subsequence $(n_k)_{k \in \mathbb N}$ and random variables $(\phi^{(k)},\partial_t \phi^{(k)},\bm X_0^{(k)}, \bm X^{(k)}, M^{(k)})_{k \geq 1}$, $(\phi,\partial_t \phi, \bm X_0,\bm X, M)$ such that $$\begin{aligned}
&(\phi^{(k)},\partial_t \phi^{(k)},\bm X_0^{(k)}, \bm X^{(k)}, M^{(k)})\\& \sim (\phi^{\epsilon_{n_k}},\partial_t \phi^{\epsilon_{n_k}}, (\phi^{\epsilon_{n_k}}_0+\epsilon_{n_k})\bm c^{\epsilon_{n_k}}_0,(\phi^{\epsilon_{n_k}} + \epsilon_{n_k})\bm c^{\epsilon_{n_k}},\int_0^\cdot (\phi^{\epsilon_{n_k}}_s+\epsilon_{n_k}) \d M^{\epsilon_{n_k}}_s)
\end{aligned}$$ and $\phi^{(k)}  \to \phi$, $\partial_s \phi^{(k)}  \to \partial_s \phi$, $\bm X^{(k)}_0 \to \bm X_0$, $\bm X^{(k)} \to \bm X$, $M^{(k)} \to M$, $\mathbb P'$-almost surely in $\mathcal A$ (cf. \eqref{CompactnessSpace}). To choose more suggestive notation, let $\bm c^{(k)} \coloneqq \frac{\bm X^{(k)}}{(\phi^{(k)} + \epsilon_{n_k})}$ and note that $\bm c^{(k)} \sim \bm c^{\epsilon_{n_k}}$. Let $$\bm c \coloneqq \mathds{1}_{\Set{\phi > 0}}\frac{\bm X}{\phi}.$$
We aim to prove that $(\phi, \bm c)$ is a solution to the martingale problem associated with equation \eqref{StationaryEquation} in the sense of Definition \ref{MartSolStationary}. To achieve this, we demonstrate that almost surely, \begin{equation} \label{StationaryVarIdentity}
\begin{aligned}
\langle \phi_t \bm c_t, \bm v \rangle &= \langle \phi_0 \bm c_0, \bm v \rangle - \int_0^t \langle \phi_s \bm D \nabla \bm c_s, \nabla \bm v \rangle \d s +  \int_0^t \langle \phi_s  f(\phi_s,\bm c_s),\bm v \rangle \d s &\\ &\quad + \int_0^t \langle \partial_s \phi_s \bm c_s, \bm v\rangle \d s+ \langle \bm v, M_t \rangle
\end{aligned}
\end{equation} for any $\bm v \in {H}^{1,2}_b$ and $t \in [0,T]$. Here, $$\nabla \bm c \coloneqq \mathds{1}_{\Set{\phi > 0}} \frac{\nabla \bm X - \bm c \otimes\nabla \phi}{\phi}$$ denotes a weak weighted gradient of the limit process $\bm c$. Subsequently, we show that $M \in {L}^2$ is a continuous, square integrable martingale with respect to the filtration generated by $(\phi|_{[0,t]},\bm c|_{[0,t]},\phi_t\bm c_t)_{t \in [0,T]}$ whose covariation is given by $$\int_0^t \phi_s b(\phi, \bm c_s)Q b^\ast(\phi,\bm c_s) \phi^\ast_s\d s.$$
Arguing as in the proof of Proposition \ref{StationaryApproximation}, we can infer $\mathbb P' \otimes \d t \otimes \d {\bm x}$-almost sure convergence $$(\phi^{(k)} + \epsilon_{n_k}) \bm c^{(k)} \to \bm X, ~\phi^{(k)} \to \phi, ~\nabla \phi^{(k)} \to \nabla \phi$$ of a suitably chosen subsequence, again denoted by $k$. This implies that $$\mathds{1}_{\Set{\phi > 0}} \bm c^{(k)} \to \mathds{1}_{\Set{\phi > 0}} \frac{\bm X}{\phi} = \bm c, ~\P' \otimes \d t \otimes \d {\bm x}\text{-a.s.}$$ and therefore $\bm c\in L^\infty(\Omega' \times [0,T] \times \mathbb T^d).$ 
We now identify the regularities and the variational equation satisfied by $\bm c$. To this end, we first observe that $\bm X^{(k)}$ converge in $L^2_w([0,T];{H}^{1,2})$, almost surely, and $\bm X^{(k)}, \mathds{1}_{\Set{\phi > 0}}\bm c^{(k)}$, $\phi^{(k)}$ and $\nabla \phi^{(k)}$ converge almost surely in $ \mathbb P'\otimes \d t \otimes \d {\bm x}$. This enables us to prove that for almost all $\omega \in \Omega'$, there exists a set $\mathcal N_\omega$ of full $\!\d t$-measure such that for any $\bm v \in {H}^{1,2}_b$ and $t \in \mathcal N_\omega$, it holds that $\bm c_t \in \mathcal X_{\bm c}$ and $$\begin{aligned}
\langle \phi_t \bm c_t, \bm v \rangle &= \langle \phi_0 \bm c_0, \bm v \rangle - \int_0^t \langle \bm D(\nabla \bm X_s -  \nabla \phi_s \otimes \bm c_s), \nabla \bm v \rangle \d s +  \int_0^t \langle \phi_s  f(\phi_s,\bm c_s),\bm v \rangle \d s \\ & \quad+\int_0^t \langle \partial_s \phi_s \bm c_s, \bm v \rangle \d s + \langle \bm v,M_t \rangle.
\end{aligned}$$ We obtain this equation by taking limits in the variational equation satisfied by $\bm X^{(k)}$ applied to $v \in {H}^{1,2}_b$. The bounded variation terms then converge since 
$$\int_0^t  \langle (\phi^{(k)}_s+\epsilon_{n_k}) \nabla \bm c^{(k)}_s, \nabla \bm v\rangle \d s = \int_0^t \langle \nabla \bm X^{(k)}_s - \bm c^{(k)}_s \otimes \nabla \phi^{(k)}_s  , \nabla \bm v \rangle \d s \to \int_0^t \langle \nabla \bm X_s - \bm c_s  \otimes \nabla \phi_s, \nabla \bm v \rangle \d s
$$ and 
$$\int_0^t \langle(\phi^{(k)}_s+\epsilon_{n_k})  f(\phi^{(k)}_s, \bm c^{(k)}_s), \bm v \rangle \d s \to \int_0^t \langle \phi_s  f(\phi_s, \bm c_s), \bm v \rangle \d s,
$$ for all $\bm v \in {H}^{1,2}$ and $t \in [0,T]$. Here, the almost sure convergence of the reaction term is a consequence of the reweighting by $(\phi^{(k)}+\epsilon_{n_k})$.

The properties of the limit process $M \in C([0,T];{L}^2)$ now follow as in the existence proof of Prop. \ref{StationaryApproximation}, where similarly we exploit the almost sure convergence of the reweighted multiplicative dispersion coefficients  $(\phi^{(k)}+\epsilon_{n_k})b(\phi^{(k)},\bm c^{(k)})$.

Continuity of the right hand side combined with uniform boundedness of $\phi_t\bm c_t \in {L}^2$ now shows that this equation holds, in fact, for all $t \in [0,T]$ and that $\phi \bm c$ is weakly continuous in ${L}^2$ - and per Lemma \ref{MySquareIto}, strongly continuous. It remains to be shown that \begin{equation} \label{WellDefinedDivision}
\nabla \bm X - \bm c \otimes\nabla \phi = \phi \cdot \mathds{1}_{\phi > 0} \frac{\nabla \bm X - \bm c \otimes \nabla \phi}{\phi}
\end{equation} and that $\nabla \bm c$ indeed defines a weighted weak derivative of $\bm c$ for almost all $t > 0$. The identity \eqref{WellDefinedDivision} quickly follows by the convergence 
$$\mathds{1}_{\Set{\phi = 0
}} \bm X = \lim_{k \to \infty} \mathds{1}_{\Set{\phi = 0
}} (\phi^{(k)} + \epsilon_{n_k})\bm c^{(k)} \equiv 0.$$ Thus, $\nabla \bm X, \nabla \phi \equiv 0$ on $\Set{\phi = 0}$, whence $\nabla \bm X - \bm c \otimes \nabla \phi = 0$ and the division is well-defined. To prove that this term defines a weighted weak derivative, let $\bm w \in (C^\infty(\mathbb T^d))^m$ be arbitrary. Then, $\!\d t$-almost surely,
$$
\begin{aligned}
\langle \phi_t \partial_i \bm c_t, \bm w \rangle = \langle \partial_i \bm X_t - \bm c_t \partial_i \phi_t, \bm w \rangle  &= - \langle \bm X_t, \partial_i \bm w \rangle - \langle \bm c_t \partial_i \phi_t, \bm w \rangle \\ &= - \langle \phi_t \bm c_t, \partial_i \bm w \rangle - \langle \bm c_t \partial_i \phi_t,\bm w \rangle.
\end{aligned}
$$
Thus, we see that $\!\d t$-almost surely, $\partial_i \bm c_t $ defines a weighted weak derivative of $\bm c$ with respect to $\phi$. However, we also need to prove that $\partial_i \bm c_t \in {L}^2(\mathbb T^d;\phi_t \d {\bm x})$. This is equivalent to proving $$\mathbb 
E'\left[\int_0^t\norm{\sqrt{\phi_s} \nabla \bm c_s}
^2_{{L}^2} \d s\right] = \mathbb E'\left[\int_0^t\frac{\norm{\nabla \bm X_s - \bm c_s \otimes\nabla \phi_s}
^2_{{L}^2}}{\phi_s} \d s\right] < \infty.$$ 
To this end, introduce a small regularisation parameter $\eta$ and observe the weak convergence  
$$
\frac1{\sqrt{\phi^{(k)}+\epsilon_{n_k}}+\eta}(\phi^{(k)}+\epsilon_{n_k})\nabla \bm c^{(k)} \rightharpoonup\frac{\nabla \bm X - \bm c \otimes\nabla \phi}{\sqrt{\phi}+\eta},
$$
which follows by weak convergence of the reweighted derivative and strong convergence of the regularised singular term. Lower semicontinuity and Fatou's lemma thus show that
$$\mathbb E'\left[\int_0^t\frac{\norm{\nabla \bm X_s - \bm c_s \otimes\nabla \phi_s}
^2_{{L}^2}}{\phi_s+\eta} \d s\right]  \leq \sup_{k \geq 1} \mathbb 
E'\left[\int_0^t\norm{\sqrt{\phi_s+\epsilon_{n_k}} \nabla \bm 
c^{(k)}_s}^2_{{L}^2} \d s\right] < \infty$$
and the claim follows by monotone convergence.
Finally, observe that the convergence and continuity properties $\phi, \bm c$ and $\phi \bm c$ imply that $\bm c_t \in \mathcal 
X_{\bm c}$ for all $t \in [0,T]$, $\mathbb P'$-
a.s.
\end{proof}

\section{Solution theory of the coupled system} 
\label{subsec:WeakSolCoupledSection}
We now aim to solve the system of coupled equations that describes the intracellular dynamics of 
biochemical components for the case of a dynamic phase-field that interacts with the dynamics of the 
stochastic reaction-diffusion equation describing the kinetics inside the phase-field. As in the 
preceding section, we modify the original system of equations by truncating the noise outside of a 
bounded hypercube $\mathcal K$ that is left invariant under the dynamics of the nonlinearity $ f$. 
The system of equations we investigate is therefore given by
\begin{equation} \label{eq:WeakSolFullSystem} 
\begin{cases}
\partial_t \phi_t &= \gamma \Delta \phi_t + g(\phi_t,\bm c_t) + \Psi(\phi_t,\bm c_t,\nabla \phi_t)
\\
\d \bm c_t &= \left(\bm D \Delta \bm c_t + \bm D \frac{\nabla \bm c
_t\nabla \phi_t}{\phi_t}  +  f(\phi_t,  \bm c_t)\right) \d t  +  b(\phi_t,\bm c_t)\d {\bm W}^Q_t,
\end{cases} 
\end{equation}
where we expanded the singular diffusion term as $$\frac{1}{\phi_t}\nabla \cdot\left(\phi_t\bm D\nabla \bm c_t\right) = \bm D \Delta \bm c_t + \bm D \frac{\nabla \bm c
_t\nabla \phi_t }{\phi_t},$$ for $\Delta \bm c_t \colon \mathbb T^d \rightarrow \mathbb R^m$ defined componentwise and $ \underbrace{\nabla \bm c_t}_{\in \mathbb R^{m \times d}}\underbrace{\nabla \phi_t}_{\in \mathbb R^d} \in \mathbb R^m$ defined as in the previous section. 

We begin this section with its main definitions and results, which posit existence of solutions of various strengths, depending on the behaviour of the initial condition $\phi_0$ near $\phi_0 \approx 0$.
\begin{remark}
To facilitate readability of theorem statements, we use an auxiliary probability space $(\Omega, \mathcal F, \mathbb P)$ to specify properties of initial distributions of martingale solutions.
\end{remark}

\begin{assumption}
Throughout all statements in this section, we assume that $\phi_0 \neq 0$ and $\bm c_0$ satisfy $$\mathbb P(\phi_0 \in \mathcal X_\phi) = \mathbb P(\bm c_0 \in \mathcal{X}_{\bm c}) =1 $$ and $\phi_0 \in L^\infty(\Omega;H^{1,2})$. In particular, we assume that $0 \leq \phi_0 \leq K_\phi$. W.l.o.g., set $K_\phi = 1$. Moreover, we again assume that $\sqrt{Q} \in HS({L}^2,{H}^{r,2})$ for some $r > \frac d 2-1 \vee 0$.
\end{assumption}
\begin{definition}[Admissible weight] \label{def:AdmissibleWeight} Let a filtered probability space $(\Omega, \mathcal F,  
(\mathcal F_t)_{t \geq 0},\mathbb P)$ and an $\mathcal F_t$-adapted phase-field 
$(\phi_t)_{t \in [0,T]} \subset H^{1,2}$ be given. A weight 
$$ 
\rho \in L^2(\Omega;L^2([0,T]; H^{1,2})) \cap L^\infty(\Omega;L^\infty([0,T];L^\infty))
$$ 
is called admissible if it is $\mathcal F_t$-adapted, absolutely continuous in $(H^{1,2}_b)^\ast$ with 
$$
\partial_t \rho \in L^2(\Omega;L^2([0,T];( H^{1,2})^\ast)) + L^1(\Omega;L^1([0,T];L^1)), 
$$ 
and satisfies the integrability condition 
$$
\mathbb E\left[\int_0^T \int \frac{|\nabla \phi_t|^2}{\phi^2_t}\rho^2_t \d {\bm x} \d t\right] < \infty. 
$$ 
\end{definition}
\begin{remark}
Any admissible weight has a  representant in $C_w([0,T];L^2)$.
\end{remark}
\begin{definition}[Weighted martingale solution, coupled case] \label{MartSolDynamic}
Fix $T > 0$, a trace class operator $Q \in L({L}^2)$ and a filtered probability space $(\Omega, \mathcal F, (\mathcal F_t)_{t \geq 0},\mathbb P)$. Let $\mathcal F_0$-measurable initial values $(\phi_0, \bm c_0) \in {L}^2((\Omega;L^2 \oplus{L}^2)$ be given. We say that $$\phi \in L^2(\Omega \times[0,T];H^{1,2}), ~\bm c \in L^2(\Omega \times [0,T];{L}^2)$$ solve the weak weighted martingale problem associated with equation \eqref{eq:WeakSolFullSystem} for noise covariance $Q$ if
\begin{enumerate}[(i)]
    \item $\phi,\bm c$ are $(\mathcal F_t)$-progressively measurable
    \item $\mathbb{P}$-almost surely, $\phi \in C_w([0,T];{L}^2)$ solves equation \eqref{DecoupledMembEq} with parameter $\bm c$
    \item $\phi \bm c\in C_w([0,T];{L}^2)$ is $\!\d {\bm x}$-weakly differentiable on a set of full $\!\d t$-measure
    \item  for all $t \in [0,T]$ and admissible weighted test functions $\bm v_t = \rho^2_t \bm u$ with $\bm u \in {H}^{1,2}_b$, $\bm c$ satisfies the variational equation 
    $$\begin{aligned}
    \langle \bm c_t, \bm v_t \rangle &= \langle \bm c_0, \bm v_0 \rangle + \int_0^t - \langle \bm D\nabla \bm c_s, \nabla \bm v_s \rangle  + \langle \bm D\frac{\nabla \bm c_s\nabla \phi_s}{\phi_s}, \bm v_s\rangle  + \langle  f(\phi_s,\bm c_s),\bm v_s \rangle  \d s 
    \\ &\quad+ \int_0^t \langle \bm v_s, \cdot \rangle  \d M_s + \int_0^t \langle c_s,\partial_s \bm v_s\rangle \d s
    \end{aligned}
    $$
    with $\rho \nabla \bm c \in L^2(\Omega \times [0,T] \times \mathbb T^d)$.
    \item Here $M$ denotes a continuous, square integrable $L^2$-valued martingale adapted to the filtration $\mathcal F_t$, with covariation $$\int_0^\cdot  b(\phi_s, \bm c_s) Q  b^\ast(\phi_s, \bm c_s)\d s.$$  
\end{enumerate} 
\end{definition}

We now state the main result of this section.
\begin{theorem}
\label{WeakSystemExistence}
Let initial data $\phi_0 \in L^\infty(\Omega;H^{1,2})$, $\bm c_0 \colon \Omega\rightarrow \mathcal X_{\bm c}$ be given with $\phi_0 \in \mathcal X_\phi$ almost surely. Then there exists a filtered probability space $(\Omega', \mathcal F', (\mathcal F'_t)_{t \geq 0}, \mathbb P')$ such that the martingale 
problem associated with equation \eqref{eq:WeakSolFullSystem} possesses a weighted solution (cf. Definition 
\ref{MartSolDynamic}) 
$$
(\phi,\bm c) \in L^\infty(\Omega'; C([0,T]; H^{1,2}) \oplus L^2([0,T];{L}^2))  
$$ 
with initial values distributed as $\phi_0 \sim \mu_{\phi_0}$, $\mathds{1}_{\Set{\phi_0>0}}\bm c_0 \sim\mu_{\mathds{1}_{\Set{\phi_0>0}}\bm c_0}$ and noise covariance 
$Q$. In particular, $\phi_t \in \mathcal X_{\phi}$ and $c_t \in \mathcal X_{\bm c}$ for all $t \in 
[0,T]$,  $\mathbb P'$-almost surely, and $$\rho^2\bm c \in L^2(\Omega';L^2([0,T]; {H}^{1,2})) \cap 
L^\infty(\Omega';C([0,T]; {L}^2))$$ for any admissible weight $\rho$.
\end{theorem}

\begin{corollary}
\label{MartingaleSolutionLog}
Assume the same conditions as in the preceding theorem and additionally, that $\log \phi_0 
\in L^1(\Omega;L^1)$. Then there exists a filtered probability space $(\Omega', \mathcal F', 
(\mathcal F'_t)_{t \geq 0}, \mathbb P')$ such that the martingale problem associated with equation \eqref{eq:WeakSolFullSystem} possesses a solution in the classical sense (Definition \ref{MarSolClassical}) with respect to the dualisation $$H^{1,2} \oplus {H}^{1,2}\hookrightarrow L^2 \oplus {L}^2 \hookrightarrow (H^{1,2} \oplus {H}^{1,2}_b)^\ast$$ with
$$
(\phi,\bm c) \in L^\infty(\Omega'; C([0,T]; H^{1,2}) \oplus C([0,T];{L}^2)) 
$$ 
and initial values distributed as $\phi_0 \sim \mu_{\phi_0}$, $\bm c_0 \sim\mu_{\bm c_0}$ and noise 
covariance $Q$.  Additionally, $\phi_t \in \mathcal X_{\phi}$, $c_t \in \mathcal X_{\bm c}$ 
for all $t \in [0,T]$,  $\mathbb P'$-almost surely.
\end{corollary}

\begin{proof}[Proof of Corollary \ref{MartingaleSolutionLog}]
If $\log \phi_0 \in L^1$, then Lemma \ref{arbitraryphibound} demonstrates that $\rho_s \equiv 1$ constitutes an admissible weight. It follows immediately that $\bm c$ is actually a martingale solution of the unweighted equation.
\end{proof}
\begin{remark}
One could have proven Corollary \ref{MartingaleSolutionLog} also directly, using the bound given by Lemma \ref{arbitraryphibound} and the compactness properties given by Lemma \ref{p1compactness}.
\end{remark}
The purpose of the following results is to elucidate the nature of the weights we employ. We note though that Proposition \ref{prop:WeightSupport} in particular will be needed to properly identify the initial value of weighted solutions. 
\begin{proposition} \label{prop:WeightSupport}
Let $\rho$ be an admissible weight. Then $\mathds{1}_{\{\phi_0=0\}}\rho_0 \equiv 0$. 
\end{proposition}
\begin{proof}
Deferred to the appendix. We note here that the proof identifies an independently interesting property; namely, the explosion of logarithmic gradient encodes information about size of the null-set of parabolic equations through the identity 
\begin{equation*}
\mu (\Set{\phi_0 > 0}) = \lim_{\alpha \to 0}\alpha \int_0^t \int \frac{|\nabla \phi_s|^2}{\phi^{2-\alpha}_s} \d \bm x \d s. \qedhere
\end{equation*}
\end{proof}

\begin{example} \label{HeatEqWeight}
Let $h$ denote the solution of the heat equation on $\mathbb T^d$ with random initial condition $h(0) = \phi_0$. Then $\sqrt{h}$ is an admissible weight. We included a proof of this fact in the appendix. 
\end{example}
\begin{example}
Lemma \ref{arbitraryphibound} below will show that in particular, positive powers $(\phi^{sub}_t)^\alpha$ of the deterministic subsolution $\phi^\text{sub}$ obtained from Proposition \ref{Subsolution} are admissible weight functions, and so is  $\phi^\alpha_t$, for any $\alpha > 0$.
\end{example}
\begin{corollary} \label{L1condition}
Let $\phi_0 \in L^\infty(\Omega;H^{1,2})$ with $\phi_0 \in \mathcal X_\phi$ almost surely. If $\rho \equiv 1$ is admissible, i.e. $\nabla \log \phi \in L^2([0,T];L^2)$, then $\phi_0 > 0$, $\d {\bm x}$-almost surely, and $\log \frac{1}{\phi_0} \in L^1$.
\end{corollary}
\begin{proof}
This follows by taking $\rho \equiv 1$ in \eqref{eq:WeightSupport}, dividing by $\alpha$ and taking $\alpha \to 0$. Since $\lim_{\alpha \to 0} \frac{1-x^\alpha}{\alpha} \to \log \frac{1}{x}$, Fatou's Lemma then gives that $$\int \log \frac{1}{\phi_0} \d {\bm x} \leq \int \log \frac{1}{\phi_t} + \int_0^T \int \frac{|\nabla \phi_t|^2} {\phi^2_t}\d {\bm x} \d s + \const,$$ for some constant $C_{\themycounter}$ independent of $\alpha$.
\end{proof}

\begin{remark}
Combined with Proposition \ref{arbitraryphibound}, we obtain the equivalance $$\phi_0 > 0 \text{ and } \log \frac{1}{\phi_0} \in L^1\iff \nabla \log \phi \in L^2([0,T];L^2)$$ for $\phi_0 \in L^\infty(\Omega;H^{1,2})$ with $0 \leq \phi_0 \leq 1$. This indicates optimality of the $L^1$ condition on the logarithm: we cannot a-priori expect any better integrability than $\nabla c \in L^2([0,T];{L}^2)$, so $\nabla \log \phi \nabla \bm c \in L^1([0,T];{L}^1)$ requires at least $\nabla \log \phi \in L^2([0,T];L^2)$.
\end{remark}

To prove Theorem \ref{WeakSystemExistence}, we again show as an intermediate step that for bounded initial data $\phi_0 \in \mathcal X_\phi$ and $\bm c_0 \in \mathcal X_{\bm c}$, there exist 
solutions to a truncated version of equation \eqref{eq:WeakSolFullSystem}, this time given by 
\begin{equation} 
\label{CoupledTruncTruncSys}
\begin{cases}
\partial_t \phi_t &= \gamma \Delta \phi_t + g(\phi_t,\bm c_t) + \Psi(\phi_t,\bm c_t,\nabla \phi_t) 
\\ \d \bm c_t &= \left(\bm D \Delta \bm c_t + \bm D \frac{\nabla \bm c_t\nabla \phi_t}{\phi_t+\epsilon}  +  f(\phi_t,  \bm c_t)\right) \d t +  b(\phi_t,\bm c_t)\d {\bm W}^Q_t
\end{cases}
\end{equation}
where $\phi_0 \sim \mu_{\phi_0}$, $\bm c_0 \sim \mu_{\bm c_0}$. Obtaining this solution will already be rather involved; we will mirror the approach of the last section and first show existence of solutions of  
\begin{equation} \label{eq:WeakSolTruncFullSystem}
\begin{cases}
\partial_t \phi_t &= \gamma \Delta \phi_t + \tilde g(\phi_t,\bm c_t) + \tilde \Psi(\phi_t,\bm c_t,\nabla^\tau \phi_t) 
\\ \d \bm c_t &= \left(\bm D \Delta \bm c_t + \bm D \frac{\nabla^\tau \bm c_t\nabla^\tau \phi_t}{\phi_t+\epsilon}  + \tilde{ f}(\phi_t,  \bm c_t)\right) \d t + \tilde{ b}(\phi_t,\bm c_t)\d {\bm W}^Q_t.
\end{cases}
\end{equation}
Here, $\nabla^\tau$ is defined as in the previous section and the truncated terms $\tilde g, \tilde \Psi, \tilde { f}$ and $\tilde{ b}$ are now truncated in both coordinates, e.g. $\tilde g(x, \bm y) = g(\tilde {x}, \tilde{\bm y})$ (cf. \eqref{eq:pointwisetruncation}).
\begin{proposition}[Existence of weak solutions of the approximating system] \label{WeakApproxSystemExistence} Assume the setting of Theorem \ref{WeakSystemExistence}. Fix $\epsilon > 0$. Then there exists a filtered probability space $(\Omega', \mathcal F', (\mathcal F'_t)_{t \geq 0}, \mathbb P')$ such that the martingale problem associated with equation \eqref{CoupledTruncTruncSys} possesses a solution $(\phi^\epsilon,\bm c^\epsilon)$ with respect to the dualisation $$H^{1,2} \oplus {H}^{1,2}\hookrightarrow L^2 \oplus {L}^2 \hookrightarrow (H^{1,2} \oplus {H}^{1,2}_b)^\ast$$ such that   $$\phi^\epsilon \in L^\infty(\Omega';C([0,T];H^{1,2})), ~\bm c^\epsilon \in L^2(\Omega';L^2([0,T];{H}^{1,2})\cap C([0,T];{L}^2)).$$  In particular, it almost surely holds that $\phi^\epsilon_t \in \mathcal X_\phi$, $\bm c^\epsilon_t \in \mathcal X_{\bm c}$ for all $t \in [0,T]$.
\end{proposition}
Before we embark on the proof of the above lemma, we investigate the properties of $\phi$.

\begin{lemma}
\label{DecoupledMembEqExistence}
Let $\phi_0 \in H^{1,2}$ and $0 < \tau < \infty$ be given. Suppose that $\phi_0 \in \mathcal X_\phi$. Given some measurable $c \colon \Omega \times [0,T] \rightarrow \mathcal X_{\bm c}$ there exists a unique variational solution $\phi$ of \begin{equation}\label{DecoupledMembEqTrunc}
\partial_t \phi_t = \gamma \Delta \phi_t + g(\phi_t,\bm c_t) +\Psi(\phi_t, \bm c_t, \nabla^\tau \phi_t)
\end{equation}
with $\phi_t \in \mathcal X_\phi$ for all $t \in [0,T]$. These solutions satisfy the additional regularity
$$\norm{\phi}_{C([0,T];H^{1,2})} < \const \text{ and } \norm{\phi}_{L^2([0,T];H^{2,2})} < C_{\themycounter}$$ for $C_{\themycounter}$ dependent on $T$, $\norm{\phi_0}_{H^{1,2}}$ and the parameters of equation \eqref{DecoupledMembEqTrunc}, but not on $\tau > 0$ or $\bm c$.
\end{lemma}
\begin{remark} \label{LimitRemark}
By a fixed point iteration or a compactness argument, it is relatively straightforward to show that these solutions of the truncated equation \eqref{DecoupledMembEqTrunc} converge to a solution of \begin{equation}\label{DecoupledMembEq}
\partial_t \phi_t = \gamma \Delta \phi_t + g(\phi_t,\bm c_t) +\Psi(\phi_t, \bm c_t, \nabla \phi_t)
\end{equation} which inherits the regularity properties of the approximations. 
\end{remark}

\begin{proof}
By standard existence theorems \cite{LiuRoeckner2015}, we immediately obtain existence of unique variational solutions of a truncated version (in $\phi$) of equation \eqref{DecoupledMembEq}. Under the boundedness condition $\phi_0 \in \mathcal X_\phi$, we obtain that $\phi_t \in \mathcal X_\phi$ for all $t \in [0,T]$ and we can remove the truncations to see that these processes are solutions of \eqref{DecoupledMembEqTrunc}. Note that this in particular uses $\varphi_i(0)=0$.

Since we assumed the higher regularity $\phi_0 \in H^{1,2}$, we can use a Galerkin approximation to rigorously justify the energy identity 
$$\frac12\norm{\nabla \phi_t}^2_{L^2} = \frac12\norm{\nabla \phi_0}^2_{L^2}+\int_0^t-\gamma \norm{\Delta \phi_s}^2_{L^2} - \langle g(\phi_s,\bm c_s),\Delta \phi_s \rangle - \langle \Psi(\phi_s, \bm c_s, \nabla^\tau \phi_s),\Delta \phi_s\rangle \d s.$$
After an application of the Young inequality for products and the Grönwall lemma, this implies that 
\begin{equation} 
\label{phiRegularity}
\phi \in C([0,T];H^{1,2}) \text{ and } \Delta \phi \in L^2([0,T];L^2).
\end{equation} By the deterministic $L^\infty$-bounds on $\phi$ and $\bm c$, the upper bounds for both quantities in \eqref{phiRegularity} are only dependent on $T$, the parameters of \eqref{DecoupledMembEq} and the norm of the initial value. 
\end{proof}

\begin{corollary}
Assume the setting of the preceding Lemma. Then, in particular, $\partial_t \phi_t \in L^2([0,T];L^2)$ is well-defined. Moreover, $\phi$ is Hölder continuous in $L^2$ with Hölder seminorm $$[\phi_t]_{C^\frac12([0,T];L^2)} \leq \norm{\partial_t\phi}_{L^2([0,T];L^2)}.$$ 
\end{corollary}
\begin{proof}
The improved regularities imply that $\Delta \phi \in L^2([0,T];L^2)$, and this regularity transfers to the distributional derivative which can now be shown to exist in $L^2$. The second claim follows now from \begin{equation*}
\norm{\phi_t-\phi_s}_{L^2} \leq \int_0^t \norm{\partial_s \phi_s} \d s \leq |t-s|^\frac12 \norm{\partial_t \phi_t}_{L^2([0,T];L^2)}. \qedhere
\end{equation*}
\end{proof}
\begin{proposition}
\label{Subsolution}
Suppose that $\phi_0 \geq 0$ and $\phi_0 \not \equiv 0$. Let $\phi_t \in L^2([0,T];H^{1,2}) \cap L^\infty([0,T];L^2)$ denote a variational solution of \eqref{DecoupledMembEq}. Then, for all $\delta > 0$, there exists $\kappa_\delta > 0$ with
$$\inf_{\delta \leq t \leq T} \essinf_{x \in \mathbb T^d} \phi(t,x) > e^{-M_1T}\kappa_\delta \int \phi_0 \d {\bm x},$$ where $M_1$ denotes the Lipschitz constant of $g$ on $[0,1]$, and in particular the function $\phi$ has full support on $\mathbb T^d$ for all positive times.
\end{proposition}

\begin{proof}
Let $M_1, M_2$ denote the Lipschitz constants of $g, \Psi$ with respect to $\phi$ and $\nabla \phi$, respectively. For nonnegative initial conditions, the solution of the partial differential equation 
$$\partial_t u_t = \gamma \Delta u_t - M_1u_t - M_2|\nabla u_t|$$ with $u_0 = \phi_0$ is a subsolution of equation \eqref{DecoupledMembEq}. Define now $v_t = e^{M_1t} u_t$ and observe that $$\partial_t v_t = \gamma \Delta v_t - M_2|\nabla v_t| = \gamma \Delta v_t + b\cdot \nabla v_t$$ for the measurable vector field
$$
b(t,x)
\coloneqq
\begin{cases}
-M_2\dfrac{\nabla  v(t,x)}{|\nabla  v(t,x)|}, & |\nabla  v(t,x)|>0,\\
0, & |\nabla v(t,x)|=0.
\end{cases}
$$
Note that $|b(t,x)|\leq M_2$. Let $\Bar{b}$ denote its periodisation to $\mathbb R^d$ and let $\Bar{\Gamma}_b(t,x;s,y)$ denote the fundamental solution of
$$
\partial_t v_t
=
\gamma\Delta v_t+\Bar{b}_t\cdot\nabla v_t
$$
on $\mathbb R^d$ \cite{AronsonFundSol}. Since the diffusion coefficient is uniformly elliptic and the drift $b$ is bounded uniformly by $M_2$, the seminal result by \citet{AronsonBounds} gives that
$$
\Bar{\Gamma}_b(t,x;s,y)
\geq
\const (t-s)^{-d/2}
\exp\left(-\const\frac{|x-y|^2}{t-s}-\const(t-s)\right)
$$
for $0\leq s<t\leq T$, with constants depending only on $T,d,\gamma$ and $M_2$.  We reduce the periodic case to this estimate. W.l.og., let $\bar x,\bar y\in[0,1)^d$ be representatives of $x,y\in\mathbb T^d$ and choose $k_0\in\mathbb Z^d$ such that
$$
|\bar x-\bar y-k_0|
\leq
\operatorname{diam}(\mathbb T^d).
$$
Then, the periodised kernel satisfies
$$
\Gamma_b(t,x;s,y)
=
\sum_{k\in\mathbb Z^d}\bar{\Gamma}_b(t,\bar x;s,\bar y+k)
\geq
\bar \Gamma(t,\bar x;s,\bar y+k_0).
$$
Thus, we obtain the Gaussian lower bound
$$
\Gamma_b(t,x;s,y)
\geq
\const (t-s)^{-d/2}
\exp\left(
-\const\frac{|\bar x-\bar y-k_0|^2}{t-s}
-\const(t-s)
\right).
$$
For $t-s\in[\delta,T]$, the right hand side is bounded from below by a strictly positive constant $\kappa_\delta$.
Now, since $v_t$ represented by this fundamental solution, we obtain for $t \geq \delta$ that
$$
v_t(x)
=
\int_{\mathbb T^d}\Gamma_b(t,x;0,y)v_0(y) \d y
\geq \kappa_\delta \int_{\mathbb T^d}u_0(y) \d y.
$$
Transforming back to $u_t=e^{-M_1t}v_t$ yields the claim.
\end{proof}

\begin{proof}[Proof of Proposition \ref{WeakApproxSystemExistence}]
Similar to the proof of Prop. \ref{prop:WeakSolStationaryTruncExistence}, existence and uniqueness of martingale solutions $(\phi^\tau, \bm c^\tau)$ of \eqref{eq:WeakSolTruncFullSystem} with respect to the Gelfand triple $$H^{1,2}\oplus {H}^{1,2} \hookrightarrow L^2 \oplus {L}^2 \hookrightarrow(H^{1,2}\oplus {H}^{1,2})^\ast$$ can be derived by Theorem 5.1.3 in \cite{LiuRoeckner2015}. The boundedness properties $\phi^\tau_t \in \mathcal X_\phi$, $\bm c^\tau_t \in \mathcal X_{\bm c}$ follow analogously to the previous sections and finally, Proposition \ref{DecoupledMembEqExistence} shows the claimed regularity $$\norm{\phi^\tau}_{L^\infty(\Omega;C([0,T];H^{1,2}))} \leq \const \norm{\phi_0}_{L^\infty(\Omega;H^{1,2})}$$ of solutions $\phi_t$, given a probability space large enough to contain the initial distributions and a countable collection of independent Brownian motions.

We now proceed by a compactness argument which shows that the distributions of these solutions are tight, given suitable initial conditions. We only sketch the argument, as the argument is similar in spirit to the proof of Theorem \ref{StationaryExistence}. By Lemma \ref{DecoupledMembEqExistence} and the subsequent corollary, we find that
$$
\phi^\tau \in L^\infty(\Omega;C([0,T];H^{1,2})\cap W^{1,2}([0,T];L^2)) 
$$ 
takes values in a compact subset of $C([0,T];L^2)$ since \begin{equation}\label{eq:compactphiembedding}
C([0,T];H^{1,2})\cap W^{1,2}([0,T];L^2) \hookrightarrow C([0,T];L^2)
\end{equation} compactly:
\begin{enumerate}[(i)]
    \item Pointwise relative compactness is clear by the compact embedding $H^{1,2} \hookrightarrow L^2$.
    \item Equicontinuity then follows from the uniform Hölder bound.
\end{enumerate}
Further, uniform boundedness of $\phi^\tau\in L^\infty([0,T];H^{1,2}) $ implies that the Itô-formula for $\norm{\bm c^\tau}^2_{{L}^2}$ can be rearranged to derive a uniform bound on $$\mathbb E\left[\int_0^T \norm{\nabla \bm c^\tau_t}^2_{{L}^2} \d t\right]< \infty.$$ By inspecting the unbounded variation term $M^\tau$ of $\bm c^\tau \in C([0,T];({H}^{1,2})^\ast)$, we finally find that $$\mathbb E \left[\norm{\bm c^\tau}_{W^{\alpha,1}([0,T];({H}^{1,2}_b)^\ast)}\right], \mathbb E \left[\norm{M^\tau}_{W^{\alpha,p}([0,T];{L}^2)}\right] < \infty$$ uniformly in $\tau$, for arbitrary $\alpha \in (0,\frac12)$, $p > 1$. 

Arguing as in Section \ref{subsec:WeakSolStationarySect}, we find that $(\mathcal L(\phi^{\tau},\bm c^{\tau}_0, \bm c^{\tau}, M^{\tau}))_{\tau > 0}$ is uniformly tight in \begin{equation}
\begin{aligned}
\mathcal A = &C([0,T];L^2) \cap L^\infty_{w^\ast}([0,T];H^{1,2}) \cap L^2_w([0,T];H^{2,2}) \times {L}^2\\ &\quad \times L^1([0,T];{L}^2)  \cap L^2_w([0,T];{H}^{1,2}) \times C([0,T];({H}^{1,2})^\ast) \cap L^\infty_{w^\ast}([0,T];{L}^2)
\end{aligned}
\end{equation} and by the Jakubowski--Skorokhod representation theorem, we obtain a probabiliy space $\Omega'$ and random variables, for the sake of simplicity again denoted $(\phi^{\tau_k},\bm c^{\tau_k}_0, \bm c^{\tau_k}, M^{\tau_k})_{n \in \mathbb N}$, such that $(\phi^{\tau_k}, \bm c^{\tau_k}_0, \bm c^{\tau_k}, M^{\tau_k}) \overset{\mathcal A}{\to} (\phi^\epsilon,\bm c^\epsilon_0, \bm c^\epsilon, M^\epsilon),$ $\mathbb P'$-almost surely for random variables $(\phi^\epsilon,\bm c^\epsilon_0, \bm c^\epsilon,M^\epsilon) \in \mathcal A$. 

Inspection of the parabolic equation satisfied by $\phi^{\tau_k}-\phi^{\tau_m}$, $k,m \geq 1$, actually yields that $(\phi^{\tau_k})_{k \geq 1}$ is Cauchy in $L^2([0,T];H^{1,2})$: To ease notation, let 
$$
\xi^{(k)}_t \coloneqq g(\phi^{\tau_k}_t,\bm c^{\tau_k}_t) + \Psi(\phi^{\tau_k}_t,\bm c^{\tau_k}_t,\nabla^{\tau_k} \phi^{\tau_k}_t) \in L^2([0,T];L^2).
$$ Then, since $(\xi^{(k)}_t)_{k \geq 1}$ is uniformly bounded in $L^2$ and $(\phi^{\tau_k})_{k \geq 1}$ is Cauchy in $L^2([0,T];L^2)$, $$\begin{aligned}
&\norm{\phi^{\tau_k}_T-\phi^{\tau_m}_T}^2_{L^2} + \gamma \int_0^T \norm{\nabla(\phi^{\tau_k}_t-\phi^{\tau_m}_t)}^2_{L^2}\d t \\
&\leq \norm{\phi^{\tau_k}_0-\phi^{\tau_m}_0}^2_{L^2} + \int_0^T (\xi^{(k)}_t-\xi^{(m)}_t)(\phi^{\tau_k}_t-\phi^{\tau_m}_t)\d t \to 0.
\end{aligned}$$ Now, as in previous proofs, the boundedness properties of the sequence imply uniform integrability and the stronger convergence $$(\phi^{\tau_k}, \bm c^{\tau_k}) \overset{L^2(\Omega';L^2([0,T];H^{1,2} \oplus {L}^2))}{\to} (\phi^\epsilon, \bm c^\epsilon).$$ holds. In particular, this means we obtain $\mathbb P' \otimes \d t \otimes \d {\bm x}$-almost sure convergence of a subsequence of $(\phi^{\tau_k}_1, \nabla \phi^{\tau_k}_1, \bm c^{\tau_k}_1)$. Without loss of generality, denote this subsequence again by $(\phi^{\tau_k}, \bm c^{\tau_k})$.

The identification of the evolution equation satisfied by $\phi$ is straightforward by these rather strong convergence and regularity properties. As in the previous section, we show that for almost every $\omega \in \Omega'$, there exists a $\d t$-null set $\mathcal N_\omega \subset [0,T]$ such that \begin{equation} \label{CoupledApproxVariationalIdentity}
\langle \bm c^\epsilon_t, \bm v \rangle = \langle \bm c^\epsilon_0, \bm v\rangle + \int_0^t- \langle \bm D\nabla\bm c^\epsilon_s, \nabla \bm v \rangle + \left \langle\bm D \frac{\nabla \bm c^\epsilon_s\nabla \phi^\epsilon_s}{\phi^\epsilon_s+\epsilon}, \bm v \right \rangle + \langle  f(\phi^\epsilon_s, \bm c^\epsilon_s), \bm v \rangle \d s + \langle M^\epsilon_t, \bm v\rangle,
\end{equation} for all $\bm v \in {H}^{1,2}_b$ and $t \notin  \mathcal N_\omega$. Except for the term involving the truncated logarithmic derivative, it is standard to identify the almost sure limits of the constituent terms of the approximation sequence $\langle \bm c^{\tau_k}, \bm v\rangle$. We now show that a subsequence of $(\tau_k)_{k \geq 1}$ exists such that almost surely, 
$$\int_0^t\int \nabla^{\tau_k} \bm c^{\tau_k}_s \frac{\nabla^{\tau_k} \phi^{\tau_k}_s}{\phi^{\tau_k}_s + \epsilon} \bm v \d {\bm x} \d s\to \int_0^t \int \nabla \bm c^\epsilon_s \frac{\nabla \phi^\epsilon_s}{\phi^\epsilon_s + \epsilon} \bm v \d {\bm x} \d s$$ for arbitrary $t > 0$ and $\bm v \in {H}^{1,2}_b,$
We decompose this integral into the sum
$$
\begin{aligned}
\int_0^t\int \nabla^{\tau_k} \bm c^{\tau_k}_s \frac{\nabla^{\tau_k} \phi^{\tau_k}_s}{\phi^{\tau_k}_s + \epsilon} \bm v \d {\bm x} \d s
&= \int_0^t\int (\nabla^{\tau_k} \bm c^{\tau_k}_s - \nabla \bm c^{\tau_k}_s)\frac{\nabla^{\tau_k} \phi^{\tau_k}_s}{\phi^{\tau_k}_s + \epsilon} \bm v \d {\bm x} \d s  \\ & \quad + \int_0^t\int \nabla \bm c^{\tau_k}_s \frac{\nabla^{\tau_k} \phi^{\tau_k}_s}{\phi^{\tau_k}_s + \epsilon} v \d {\bm x} \d s
\end{aligned}
$$
Convergence of the second integral follows as $\nabla \bm c^{\tau_k}$ almost surely converges weakly, while $\frac{\nabla^{\tau_k} \phi^{\tau_k}}{\phi^{\tau_k}_t+\epsilon}$ converges strongly. To show the the first integral vanishes, we notice that for $\nu > 1$,
$$\begin{aligned}
\mathbb E\left[\left| \int_0^t\int (\nabla^{\tau_k} \bm c^{\tau_k}_s - \nabla \bm c^{\tau_k}_s)\frac{\nabla^{\tau_k} \phi^{\tau_k}_s}{\phi^{\tau_k}_s + \epsilon} \bm v \d {\bm x} \d s\right| \right]&\leq \frac{\norm{\bm v}_{{L}^\infty}}{\epsilon} \mathbb E\left[\int_0^T \int |\nabla^{\tau_k} \bm c^{\tau_k}_s - \nabla \bm c^{\tau_k}_s| |\nabla^{\tau_k} \phi^{\tau_k}| \d {\bm x} \d t \right] \\
& \leq \frac{\const\norm{\bm v}_{{H}^{1,2}_b}}{\epsilon} \norm{\nabla^{\tau_k} \bm c^{\tau_k}_s - \nabla \bm c^{\tau_k}_s}_{L^{\frac{\nu}{\nu-1}}} \norm{\nabla \phi^{\tau_k}}_{L^{\nu}},
\end{aligned}$$ 
where the $L^p$-norms are taken with respect to probability, space and time. We need to verify that $\norm{\nabla \phi^{\tau_k}}_{L^{\nu}}$ is finite for relevant values of $\nu$. Choose $\nu > 2$, so that $\frac{\nu}{\nu-1}<2$. Without loss of generality, assume that $d \geq 2$; the corresponding argument for $d = 1$ follows by the stronger Sobolev embeddings. Importantly, for $\nu \leq \frac{2d}{d-2}$, the Gagliardo-Nirenberg inequality implies that 
$$
\norm{\nabla \phi_t}_{L^\nu} \leq \norm{\Delta \phi_t}^{\theta_\nu}_{L^2} \norm{\nabla \phi_t}^{1-\theta_\nu}_{L^2}
$$ with $\theta_\nu =d(\frac12-\frac1\nu) \to 0$ as $\nu \to 2$. This gives the desired regularity: We have $$\phi_t \in L^\infty(\Omega;L^\infty([0,T];H^{1,2}) \cap L^2([0,T];H^{2,2})),$$ and thus $\nabla \phi_t \in L^\infty(\Omega;L^\frac{2}{\theta_\nu}([0,T];L^\nu))$, with $\frac{2}{\theta_\nu} \to \infty$ as $\nu \to 2$. Moreover, by the uniform a-priori bounds, the Markov inequality shows that  $$(\mathbb P'\otimes \d t \otimes \d {\bm x})\left(\Set{ |\nabla \bm c^{\tau_k}| \geq {\tau_k}}\right) \leq \frac{\const}{{\tau_k}^2}$$ and by that can conclude that
$$
\begin{aligned}
 &\mathbb E' \left[\int_0^T \int |\nabla^{\tau_k} \bm c^{\tau_k}_t- \nabla \bm c^{\tau_k}_t|^{\frac{\nu}{\nu-1}} \d {\bm x} \d t \right]\\ &= \mathbb E' \left[\int_0^T \int (|\nabla \bm c^{\tau_k}_t|-{\tau_k})^{\frac{\nu}{\nu-1}} \mathds{1}_{\Set{|\nabla \bm c^{\tau_k}| \geq {\tau_k}}} \d {\bm x} \d t \right]
\\ & \leq \mathbb E' \left[ \norm{\nabla \bm c^{\tau_k}}^2_{L^2([0,T] \times \mathbb T^d)}\right]^\frac q2 \cdot (\mathbb P' \otimes \d t \otimes \d {\bm x})^{\frac{\nu-2}{2(\nu-1)}}\left(\Set{|\nabla \bm c^{\tau_k}| \geq {\tau_k}}\right).
\end{aligned}
$$
The right hand side converges to $0$, which in particular proves convergence of the sequence of integrals to $0$ in $L^{1}(\Omega')$. Thus we can choose a subsequence such that this expression converges almost surely to $0$ and we see that the limit satisfies the right expression.

The properties of the martingale $M^\epsilon_t$ follow as in previous proofs, and similarly does the claimed regularity.
\end{proof}
\begin{lemma} \label{arbitraryphibound}
For $\bm c \in L^\infty(\Omega \times [0,T]\times \mathbb T^d)$ and an initial condition $\phi_0$ with $ \mathbb P(\phi_0 \in \mathcal X_\phi) = 1$, $$\alpha \int_0^T \norm{\frac{|\nabla \phi_t|}{\phi^{1-\alpha}_t}}^2_{L^2} \d t \in L^\infty(\Omega)$$ uniformly in $\alpha \in (0,\frac12)$. In particular, $(\phi^\alpha)_{\alpha \in (0,\frac12)} \subset L^\infty(\Omega;L^2([0,T]; H^{1,2}))$ is uniformly bounded. If, additionally, $\log \phi_0 \in L^1(\Omega;L^1),$ then 
$$
\sup_{0 \leq t \leq T} \norm{\log \phi_t}_{L^1}+\int_0^T \norm{\frac{|\nabla \phi_t|}{\phi_t}}^2_{L^2} \d t \in L^1(\Omega).
$$ 
\end{lemma}

\begin{proof}
We prove the statement with the auxiliary variable $0 < \beta \coloneqq 2\alpha < 1$. By the chain rule, we obtain the identity
$$
\begin{aligned}
&\int (1-(\phi_t+\epsilon)^\beta)\d {\bm x} - \int (1-(\phi_0+\epsilon)^\beta) \d {\bm x} \\&= -\beta\left(-(1-\beta)\int_0^t \int \frac{|\nabla \phi_s|^2}{(\phi_s+\epsilon)^{2-\beta}}\d {\bm x} \d s + \int_0^t \int  \frac{\Psi(\phi_s, \bm c_s, \nabla \phi_s)}{(\phi_s+\epsilon)^{1-\beta}} \d {\bm x} + \int_0^t \int \frac{g(\phi_s, \bm c_s)}{(\phi_s+\epsilon)^{1-\beta}} \d {\bm x} \d s  \right).
\end{aligned}
$$
Now since $|\Psi(\phi_s, \bm c_s, \nabla \phi_s)| \leq \const |\nabla \phi_s|$, we observe that $$\left \lvert \int  \frac{\Psi(\phi_s, \bm c_s, \nabla \phi_s)}{(\phi_s+\epsilon)^{1-\beta}} \d {\bm x}\right \rvert \leq C_{\themycounter}\int  \frac{\lvert\nabla \phi_s\rvert}{(\phi_s+\epsilon)^{1-\beta}} \d {\bm x} \leq \frac{1-\beta}{2} \int  \frac{\lvert\nabla \phi_s\rvert^2}{(\phi_s+\epsilon)^{2-\beta}} \d {\bm x} + \frac{2C^2_{\themycounter}}{1-\beta} \int (\phi_s+\epsilon)^\beta \d {\bm x}.$$ Further, the boundedness of $g(\phi_s, \bm c_s)/\phi_s$ (cf. Assumption \ref{ass:WeakSolNemytskiiAssumption}) implies boundedness of the third term on the right hand side. Thus, after rearrangement, we find that 
$$
\beta \int_0^t \int \frac{|\nabla \phi_s|^2}{(\phi_s+\epsilon)^{2-\beta}}\d {\bm x} \d s 
=  \beta \int_0^T  \norm{\frac{|\nabla \phi_t|}{\phi^{1-\beta/2}_t}}^2_{L^2}  
\d t  \leq \const, 
$$ 
for some constant $C_{\themycounter}$ independent of $\epsilon$. Application of Fatou's Lemma to take the limit $\epsilon \to 0$ concludes the proof. If, additionally, $\log \phi_0 \in L^1$, then we can repeat the above computation for $\log \frac{1}{\phi_t + \epsilon}$ and rearrange accordingly to obtain the claimed a priori estimate.
\end{proof}

\begin{lemma} \label{LastApriori}
Let a filtered probability space $(\Omega, \mathcal F, (\mathcal F_t)_{t \geq 0}, \mathbb P)$ supporting a collection of solutions $(\phi^\epsilon, c^{\epsilon})_{\epsilon>0}$ of the Martingale problem associated with equation \eqref{CoupledTruncTruncSys} be given. For any fixed $\alpha \in (0,\frac12)$, there exists $p > 1$  such that $((\phi^\epsilon +\epsilon)^\alpha\bm c^\epsilon)_{\epsilon > 0}$ is uniformly bounded in  $$L^2(\Omega,L^2([0,T];{H}^{1,2})) \cap L^p(\Omega; W^{\beta,p}([0,T];({H}^{1,2}_b)^\ast))$$ and  $$L^p(\Omega;W^{1,p}([0,T];({H}^{1,2}_b)^\ast)) + L^q(\Omega;W^{\beta,q}([0,T];({H}^{1,2}_b)^\ast)),$$ respectively, for any $\beta \in (0,\frac12)$ and $q > \frac{1}{\beta}$. If, additionally, $\log \phi_0 \in L^\infty(\Omega;L^1)$, then $$(c^\epsilon)_{\epsilon > 0} \subset L^1(\Omega;L^2([0,T];{H}^{1,2}) \cap W^{\alpha,1}([0,T];({H}^{1,2}_b)^\ast))$$ is uniformly bounded. 
\end{lemma}
\begin{proof}
The second part of the statement follows by the previous Lemma combined with the usual a-priori inequality one obtains by rearranging the Itô formula applied to $\norm{\bm c^\epsilon_t}^2_{{L}^2}$.

We obtain the remaining a-priori bounds through the Itô formula applied to $(\phi^\epsilon+\epsilon)^\alpha \bm c^\epsilon$. A first application Lemma \ref{MyProdIto} yields the variational identity
\begin{equation} \label{phiweighting}
\begin{aligned}
&\d \langle (\phi_t^\epsilon+\epsilon)^\alpha\bm c^\epsilon_t, \bm v \rangle \\ &= \Big(- \langle \bm D (\phi_t^\epsilon+\epsilon)^{\alpha}\nabla \bm c^\epsilon_t, \nabla v  \rangle - \langle ((\bm D+ \gamma \bm I)\alpha-\bm D) (\phi_t^\epsilon+\epsilon)^{\alpha-1}   \nabla \bm c^\epsilon_t\nabla \phi^\epsilon_t , \bm v \rangle  \\ &\quad + \langle  (\phi_t^\epsilon+\epsilon)^{\alpha} f(\phi^\epsilon_t,\bm c^\epsilon_t), \bm v \rangle  - \alpha \gamma \langle (\phi_t^\epsilon+\epsilon)^{\alpha-1}  \bm c^\epsilon_t \otimes \nabla \phi^\epsilon_t , \nabla \bm v \rangle
\\ & \quad - \alpha(\alpha-1)\gamma \left \langle \frac{|\nabla \phi^\epsilon_t|^2}{(\phi^\epsilon_t+\epsilon)^{2-\alpha}} \bm c^\epsilon_t, \bm v\right \rangle + \alpha \langle (\phi_t^\epsilon+\epsilon)^{\alpha-1} \Psi(\phi^\epsilon_t, \bm c^\epsilon_t,\nabla \phi^\epsilon_t) \bm c^\epsilon_t, \bm v \rangle \\ & \quad + \alpha \langle (\phi_t^\epsilon+\epsilon)^{\alpha-1} g(\phi^\epsilon_t) \bm c^\epsilon_t , \bm v \rangle\Big) \d t + \langle \bm v, (\phi_t^\epsilon+\epsilon)^{\alpha} b(\phi^\epsilon_t,\bm c^\epsilon_t) \d {\bm W}^Q_t \rangle .
\end{aligned}.\end{equation}
for any $\bm v \in {H}^{1,2}_b$. Note that the quadratic covariation vanishes in this formula since $\phi^\epsilon$ is of bounded variation. With this variational identity at hand, we can apply the Itô type formula given by Lemma \ref{MySquareIto} to derive that
$$
\begin{aligned}
\frac12 \d \norm{(\phi_t + \epsilon)^{\alpha/2}\bm c^\epsilon_t}^2_{{L}^2} &=  \Bigg( - \norm{ \sqrt{\bm D}(\phi_t + \epsilon)^{\alpha/2} \nabla \bm c^\epsilon_t}^2_{{L}^2 \times {L}^2} - \langle ((\bm D+ \gamma \bm I)\alpha-\bm D) (\phi_t^\epsilon+\epsilon)^{\alpha-1}  \nabla \bm c^\epsilon_t \nabla \phi^\epsilon_t, \bm c^\epsilon_t \rangle  \\ & \quad + \langle  (\phi_t^\epsilon+\epsilon)^{\alpha} f(\phi^\epsilon_t, \bm c^\epsilon_t), \bm c^\epsilon_t \rangle  - \alpha \gamma \langle (\phi_t^\epsilon+\epsilon)^{\alpha-1}  \bm c^\epsilon_t \otimes \nabla \phi^\epsilon_t , \nabla \bm c^\epsilon_t \rangle
\\ & \quad - \frac12 \alpha(\alpha-1)\gamma \left \langle \frac{|\nabla \phi^\epsilon_t|^2}{(\phi^\epsilon_t+\epsilon)^{2-\alpha}} \bm c^\epsilon_t, \bm c^\epsilon_t \right \rangle + \frac12 \alpha \langle (\phi_t^\epsilon+\epsilon)^{\alpha-1} \Psi(\phi^\epsilon_t, \bm c^\epsilon_t,\nabla \phi^\epsilon_t) \bm c^\epsilon_t, \bm c^\epsilon_t \rangle \\ & \quad + \frac12 \alpha \langle (\phi_t^\epsilon+\epsilon)^{\alpha-1} g(\phi^\epsilon_t) \bm c^\epsilon_t , \bm c^\epsilon_t \rangle + \frac12 \sum_{k \geq 0} \lambda_n \norm{ (\phi_t+\epsilon)^{\alpha/2} b(\phi^\epsilon_t, \bm c^\epsilon_t) \bm e_n}^2_{{L}^2} \Bigg)\d t
\\&\quad+ \langle(\phi_t+\epsilon)^{\alpha} \bm c^\epsilon_t,   b(\phi^\epsilon_t, \bm c^\epsilon_t) \d {\bm W}^Q_t \rangle .
\end{aligned}
$$ After rearranging this equation, we exploit uniform boundedness of $\phi^\epsilon \in \mathcal X_\phi$, $\bm c^\epsilon \in \mathcal X_{\bm c}$ and $$\norm{\frac{\nabla \phi^\epsilon}{(\phi^\epsilon+\epsilon)^{1-\alpha/2}}}_{L^2}$$ to show that \begin{equation} \label{pboundlogderiv}
\mathbb E\left[\int_0^T\norm{(\phi_t + \epsilon)^{\alpha/2} \nabla \bm c^\epsilon_t}^2_{{L}^2 \times {L}^2} \right] \leq K_\alpha
\end{equation} for some constant $K_\alpha$ dependent on $\alpha$ but not on $\epsilon$. Obtaining the bound on the right hand side of the equation is tedious but mostly standard; the only noteworthy step is the estimate $$\begin{aligned}
\lvert\langle (\phi_t^\epsilon+\epsilon)^{\alpha-1}  \nabla \bm c^\epsilon_t \nabla \phi^\epsilon_t, \bm c^\epsilon_t \rangle\rvert &\leq \norm{c^\epsilon_t}_{{L}^\infty}\langle (\phi_t^\epsilon+\epsilon)^{\alpha/2-1}|\nabla \phi^\epsilon_t|, (\phi_t^\epsilon+\epsilon)^{\alpha/2}|\nabla \bm c^\epsilon_t| \rangle \\ &\leq \norm{c^\epsilon_t}_{{L}^\infty}\left(C_\delta \norm{\frac{\nabla \phi^\epsilon}{(\phi^\epsilon+\epsilon)^{1-\alpha/2}}}_{L^2}^2 + \delta \norm{(\phi_t + \epsilon)^{\alpha/2} \nabla \bm c^\epsilon_t}^2_{{L}^2 \times {L}^2} \right),
\end{aligned}$$
where we can choose $\delta \ll 1$ small enough to be absorbed by $\bm D$, owing to Young's inequality. As a consequence of the deterministic bounds on the norm of $(\phi^\epsilon+\epsilon)^\alpha \in L^2([0,T];H^{1,2})$, we know that $$\bm c^\epsilon \otimes  \nabla (\phi^\epsilon+\epsilon)^\alpha \in L^\infty(\Omega;L^2( [0,T];{L}^2))$$ is uniformly bounded in $\epsilon > 0$ and the desired estimate on the ${L}^{2}$-norm of $$\nabla ((\phi^\epsilon+\epsilon)^\alpha\bm c^\epsilon) = (\phi^\epsilon+\epsilon)^\alpha \nabla \bm c^\epsilon + \nabla(\phi^\epsilon+\epsilon)^\alpha \otimes \bm c^\epsilon$$ follows.

With this estimate at hand, we can control the norm of $((\phi^\epsilon+\epsilon)^\alpha \bm c^\epsilon)_{\epsilon > 0}$ in the space and $$W^{1,p}([0,T];({H}^{1,2}_b)^\ast) + W^{\beta,q}([0,T];({H}^{1,2}_b)^\ast)$$ for small $p > 1$ and any $q>2$ and $\beta \in (\frac1q,\frac12)$. Here, the unbounded variation part of $(\phi^\epsilon+\epsilon)^\alpha \bm c^\epsilon$ lives in the former space, while the stochastic integral lives in the latter. As this estimate additionally yields boundedness in $$W^{\beta,p}([0,T];({H}^{1,2}_b)^\ast),$$ we can apply both embeddings given by Lemma \ref{p1compactness}. The claimed estimate on the norm of $(\phi^\epsilon+\epsilon)^\alpha \bm c^\epsilon$ in $W^{1,p}$ follows by a (mostly standard) dual space estimate applied to the individual terms in \eqref{phiweighting}. However, once again, the estimate on the logarithmic derivative is quite involved. To derive the estimate in $W^{1,p}([0,T];({H}^{1,2}_b)^\ast)$, we show that $$\mathbb E\left[\int_0^T \norm{ (\phi_t^\epsilon+\epsilon)^{\alpha/2-1}|\nabla \phi^\epsilon_t|(\phi_t^\epsilon+\epsilon)^{\alpha/2}|\nabla \bm c^\epsilon_t|}^p_{L^p} \d t\right] < \infty$$ uniformly in $\epsilon > 0$ for small enough $p > 1$. The Young inequality for products, applied to $q = \frac{2}{p}$ and $q' = \frac{2}{2-p}$ shows that $$\int_0^T \norm{ (\phi_t^\epsilon+\epsilon)^{\alpha/2-1}|\nabla \phi^\epsilon_t|(\phi_t^\epsilon+\epsilon)^{\alpha/2}|\nabla \bm c^\epsilon_t|}^p_{L^p} \d t\leq \int_0^T \frac{1}{q}\norm{(\phi^\epsilon+\epsilon)^{\alpha/2} |\nabla c^\epsilon_t|}^2_{L^2}+\frac{1}{q'}\norm{\frac{|\nabla \phi^\epsilon|}{(\phi^\epsilon_t + \epsilon)^{1-\alpha/2}}}^{pq'}_{L^{pq'}} \d t.$$ Now, for unspecified $0 < a < \alpha/2$, decompose $$\frac{|\nabla \phi^\epsilon|}{(\phi^\epsilon_t + \epsilon)^{1-\alpha/2}} = \frac{|\nabla \phi^\epsilon|^{(1-a)}}{(\phi^\epsilon_t + \epsilon)^{1-\alpha/2}} |\nabla \phi^\epsilon|^a$$ and for $\nu > 2$, apply the Hölder inequality with $\tilde{q} = \frac{\nu}{apq'}$ to derive that $$\norm{\frac{|\nabla \phi^\epsilon|}{(\phi^\epsilon_t + \epsilon)^{1-\alpha/2}}}^{pq'}_{L^{pq'}} \leq \norm{\nabla \phi}^\frac{\nu}{\tilde q}_{L^\nu} \norm{\frac{|\nabla \phi^\epsilon|^{(1-a)pq'}}{(\phi^\epsilon_t + \epsilon)^{(1-\alpha/2)pq'}}}_{L^{\frac{\tilde q}{\tilde q - 1}}} \leq \const \left(\norm{\nabla \phi}^\nu_{L^\nu}+\norm{\frac{|\nabla \phi^\epsilon|^{(1-a)pq'}}{(\phi^\epsilon_t + \epsilon)^{(1-\alpha/2)pq'}}}^{\frac{\tilde q}{\tilde q - 1}}_{L^{\frac{\tilde q}{\tilde q - 1}}} \right) .$$ Finally, choose $p = \frac{2\nu}{2\nu - (\nu-2)a}$. A straightforward but tedious computation shows that for $a < \alpha/2 < 1$ and $\nu > 2$, it holds that $1<p<\infty$, $pq'a < \nu$ and that $$(1-\alpha/2) p q'\frac{\tilde q}{\tilde q - 1} < (1-a)pq' \frac{\tilde q}{\tilde q - 1}= 2.$$ As in the proof of Proposition \ref{WeakApproxSystemExistence}, we derive $L^\nu(\Omega \times [0,T]\times \mathbb T^d)$-regularity of $\phi^\epsilon$ for $\nu$ close to $2$. We can conclude from Lemma \ref{arbitraryphibound} that \ref{pboundlogderiv} holds for this exponent.
\end{proof}
\begin{proof}[Proof of Theorem \ref{WeakSystemExistence}] We divide the proof into several steps. We aim to derive existence of processes $(\phi,\bm c)$ such that for some sequence $\alpha_n \to 0$, $(\phi^{\alpha_n} c)_{n \geq 1}$ satisfy a variational identity. Then, for any admissible weight $\rho$, we can derive a variational identity for $\rho^2 \bm c = \lim_{\alpha \to 0} \rho^2 \phi^\alpha \bm c$.

To derive existence of these solutions and the properties of $(\phi^{\alpha_n} \bm c)_{n \geq 0}$, we again employ a compactness argument. For fixed $\epsilon_n$, consider the sequence $X_n = (x^n_m)_{m \geq 1}$ defined by $$(X_n)_{n \geq 1} \coloneqq (\phi^{\epsilon_n}, M^{\epsilon_n},((\phi^{\epsilon_n}+\epsilon_n)^{\frac1k}\bm c^{\epsilon_n})_{k \geq 1}, (\bm c^{\epsilon_n}|_{[T/k, T]})_{k \geq 1})_{n \geq 1} \subset \mathcal A \coloneqq \mathcal B_1 \times \mathcal B_2 \times \prod_{k = 1}^\infty \mathcal C \times \prod_{k = 1}^\infty \mathcal C_k,$$ where $M^\epsilon(t) \in {L}^2$ is the martingale part of $\bm c^{\epsilon_n}$. Here, \begin{enumerate}
    \item $\mathcal B_1 = C([0,T];L^2) \cap L^\infty_{w^\ast}([0,T];H^{1,2}) \cap L^2_w([0,T];H^{2,2})$,
    \item $\mathcal B_2 = C([0,T];({H}^{1,2})^\ast) \cap L^\infty_{w^\ast}([0,T];{L}^2)$,
    \item $\mathcal C = L^2([0,T];{L}^2) \cap C([0,T];({H}^{1,2}_b)^\ast) \cap L^2_w([0,T];{H}^{1,2})$,
    \item and $\mathcal C_k = L^2([T/k,T];{L}^2) \cap C([T/k,T];({H}^{1,2}_b)^\ast)\cap L^2_w([T/k,T];{H}^{1,2}).$
\end{enumerate} 

\textbf{Claim 1:} The laws of $(X_n)_{n \geq 1}$ are uniformly tight in $\mathcal A$ under the product topology.

By virtue of this claim, we can apply the almost sure Skorokhod representation theorem given in \cite{JakubowskiSkorokhod}, as the topological condition specified therein is stable under countable products. Let $1 \geq \epsilon_n \overset{n \to \infty}{\to} 0$ be a vanishing sequence of real numbers.

We prove uniform tightness by constructing compact sets in each coordinate with summable exceptional probabilities. More precisely, we derive uniform tightness estimates in Banach spaces $(\mathcal Y_m)_{m \geq 1}$ such that an embedding $$\iota \colon \prod_{m \geq1} \mathcal Y_m \hookrightarrow \mathcal A$$ exists with each coordinate mapping a compact embedding. For all coordinates except the positive-time restrictions $\bm c^{\epsilon_n}|_{[T/k,T]}$, the compactness argument is direct. To be more explicit:
\begin{enumerate}[(1)]
    \item We know from Lemma \ref{DecoupledMembEqExistence} that we can apply the compact embedding \eqref{eq:compactphiembedding}, so  $(\phi^\epsilon)_{\epsilon > 0}$ is compact in $C([0,T]; L^2)$.
    \item In the preceding Lemma, we derived uniform bounds in mean on $((\phi^{\epsilon}+\epsilon)^{\frac1k}\bm c^{\epsilon_n})_{\epsilon > 0}$ for each $k \geq 1$. The bounds are in spaces which satisfy the desired compact embedding into $\mathcal C$.
    \item Finally, due to the fact that there exists a deterministic upper bound on the Nemytskii operator $ b(\phi^\epsilon_s, \bm c^\epsilon_s) \in L^\infty([0,T];L^\infty)$, it follows that for any $\alpha \in (0,\frac12)$ and $p > 1/\alpha$, $(M^\epsilon)_{\epsilon > 0}$ is uniformly bounded in mean in $W^{\alpha, p}([0,T];{L}^2) \hookrightarrow C^{\alpha - 1/p}([0,T];{L}^2)$ (cf. the proof of Proposition \ref{stationaryApriori}).
\end{enumerate}

It remains to discuss the coordinates of the type $\bm c^{\epsilon_n}|_{[T/k,T]}$. For $a>0$, define the event
\begin{equation} \label{eq:PhiLowerBoundLocalisation}
E_a
\coloneqq
\Set{\int \phi_0 \d \bm x\geq a}.
\end{equation}
The quantitative subsolution estimate from Proposition \ref{Subsolution} yields a deterministic constant $\kappa_{k,a}>0$, independent of $n$, such that
$$
\essinf_{(t,x)\in[T/k,T]\times\mathbb T^d}\phi^{\epsilon_n}(t,x)
\geq
\kappa_{k,a}\int_{\mathbb T^d}\phi^{\epsilon_n}_0(x)\,\d {\bm x}
$$
almost surely. Consequently, on $E_a$,
$
\norm{1/\phi^{\epsilon_n}}_{L^\infty([T/k,T]\times\mathbb T^d)}
\leq
\frac{1}{\kappa_{k,a}}.
$
Since the laws of $\phi^{\epsilon_n}_0$ do not depend on $n$ and
$
\mathbb P\left(\int_{\mathbb T^d}\phi_0(x)\d \bm x>0\right)=1,
$
we have
$$
\lim_{a\searrow 0}
\sup_{n\geq1}
\mathbb P\left(\phi_0 \notin E_a \right)
=
0.
$$
On $E_a$, the equation for $\bm c^{\epsilon_n}$ is nonsingular on $[T/k,T]$ with constants depending on $k$ and $a$, but not on $n$. Moreover, since $E_a$ depends only on the initial value $\phi_0$, localisation is applicable. Thus, uniform bounds in mean on $$(\mathds{1}_{E_a}\bm c^{\epsilon}|_{[T/k, T]})_{\epsilon > 0}$$ can be obtained analogously (and in the same spaces, modulo time shift) to those specified in the statement of Lemma \ref{LastApriori} in the case $\log \phi_0 \in L^1$.
In particular, since $\nabla \phi^\epsilon \in L^\nu(\Omega \times [0,T] \times \mathbb T^d)$ for some $\nu > 2$, $\nabla \bm c^{\epsilon}\nabla \phi^\epsilon \in L^q$ for $q = \frac{2\nu}{\nu+2} > 1$, we can use the second statement of Lemma \ref{p1compactness} to obtain uniform convergence in the dual space $({H}^{1,2}_b)^\ast$. 

Let $\eta > 0$. Consider first coordinates not of the form $\bm c^{\epsilon_n}|_{[T/k,T]}$. By Markov's inequality, we can choose a radius $R_j>0$ so large that
$$
\sup_{n\geq1}
\mathbb P\left(
\norm{x^n_j}_{\mathcal Y_j}>R_j
\right)
\leq
\frac{\eta}{2^{j+1}}.
$$ In the other case, choose first $a_j>0$ so small that
$
\sup_{n\geq1}
\mathbb P\left(\phi_0 \notin E_{a_j}\right)
\leq
\frac{\eta}{2^{j+2}}.
$
Subsequently, choose $R_j>0$ so large that by Markov's inequality,
$$
\sup_{n\geq1}
\mathbb P\left(
E_{a_j}
\cap
\left\{
\norm{\bm c^{\epsilon_n}}_{\mathcal Y_j}>R_j
\right\}
\right) = \sup_{n\geq1}
\mathbb P\left(
\mathds{1}_{E_{a_j}}
\norm{\bm c^{\epsilon_n}}_{\mathcal Y_j}>R_j
\right) 
\leq
\frac{\eta}{2^{j+2}}.
$$
Then $$\sup_{n\geq1}
\mathbb P\left(
\norm{\bm c^{\epsilon_n}}_{\mathcal Y_j}>R_j
\right) \leq \sup_{n\geq1}
\mathbb P\left(\phi_0 \notin E_{a_j}\right)+\sup_{n\geq1}
\mathbb P\left(
\mathds{1}_{E_{a_j}}
\norm{\bm c^{\epsilon_n}}_{\mathcal Y_j}>R_j
\right) \leq \frac{\eta}{2^{j+1}}.$$
From this, we can derive uniform tightness rather quickly: let 
$$K_\eta \coloneqq \prod_j  \overline{\iota B_{\mathcal Y_j}(R_j)} \subset \mathcal A.$$ By compactness of the embeddings and Tykhonoff's theorem, this set is compact in $\mathcal A$, and $$\begin{aligned}
\mathbb P(X_n \notin K_\eta) \leq \sum_{j \geq 1}\mathbb P(x^n_j \notin B_{\mathcal Y_j}(R_j)) \leq \sum_{j \geq1} \frac{\eta}{2^j}  = \eta,
\end{aligned}$$
which demonstrates uniform tightness.  By the Skorokhod representation theorem, we obtain a probabiliy space $(\tilde \Omega,\tilde {\mathcal F},\tilde {\mathbb P})$ and random variables, for the sake of simplicity again denoted $(X_n)_{n \geq 1}$, such that $X_n \overset{\mathcal A}{\to} X $, $\tilde{\mathbb P}$-almost surely.

\textbf{Claim 2}: We now claim that there exist random variables $\phi, \bm c$ such that
\begin{enumerate}
    \item[(i)] $\phi^{\epsilon_n} \to \phi$ in $C([0,T];L^2) \cap L^\infty_{w^\ast}([0,T];H^{1,2}) \cap L^2_w([0,T];H^{2,2})$.
    \item[(ii)] $\bm c^{\epsilon_n} \to \bm c$ in $L^2([0,T];{L}^2)$ and in $C([s,T];({H}^{1,2}_b)^\ast)$ for any $s > 0$.
    \item[(iii)] $\nabla \bm c^{\epsilon_n} \to \nabla \bm c$ weakly in $L^1([s,T];{L}^2)$ for all $s > 0$.
    \item[(iv)] $\bm (\phi^{\epsilon_n}+\epsilon_n)^\alpha \bm c^{\epsilon_n} \to \phi^\alpha\bm c$ weakly in $L^2([0,T];{H}^{1,2}) \cap C([0,T];({H}^{1,2})^\ast)$ for all $\alpha = \frac1k$,  $k \in \mathbb N$.
    \item[(v)] $(\phi^{\epsilon_n}+\epsilon_n)^\alpha \nabla \bm c^{\epsilon_n} \to \phi^\alpha \nabla \bm c$ weakly in $L^2([0,T];{L}^2)$ for all $\alpha = \frac1k$,  $k \in \mathbb N$..
    \item[(vi)] $M^{\epsilon_n} \to M$ in $C([0,T];({H}^{1,2})^\ast) \cap L^\infty_{w^\ast}([0,T];{L}^2)$.
    \item[(vii)] $\phi^{\epsilon_n} \to \phi$, $\nabla \phi^{\epsilon_n} \to \nabla \phi$  and $\bm c^{\epsilon_n} \to \bm c$ $\tilde{\mathbb P} \otimes \d t \otimes \d {\bm x}$-almost surely, on a subsequence again denoted by $(\phi^{\epsilon_n},\bm c^{\epsilon_n})_{n \geq 1}$.
\end{enumerate}
Claims (i), (ii), (iii), (vi) and (iv) and (v) essentially follow from convergence of $(X_n)_{n \geq 1} \subset \mathcal A$ in combination with (vii). In particular, we obtain a limit function $\bm c$ since the estimate $$\norm{\bm c^{\epsilon_n}-\bm c^{\epsilon_k}}_{L^2([0,T];{L}^2)} \leq \underbrace{\norm{\bm c^{\epsilon_n}-\bm c^{\epsilon_k}}_{L^2([T/m,T];{L}^2)}}_{\to 0} + \underbrace{\norm{\bm c^{\epsilon_n}- \bm c^{\epsilon_k}}_{L^2([0,T/m]{L}^2)}}_{\leq \const/\sqrt{m}} \to 0$$ shows that $(c^{\epsilon_n})_{n \geq 1}$ is a Cauchy sequence. 

As $(c^{\epsilon_n})_{n \geq 1} \subset L^\infty(\tilde \Omega \times [0,T] \times \mathbb T^d)$ is uniformly bounded, the above estimate translates into $L^2(\tilde \Omega \times [0,T]\times\mathbb T^d)$-convergence due to the Lebesgue DCT. This, in turn, implies almost sure convergence of a subsequence. We can repeat this argument for $(\phi^{\epsilon_n})_{n \geq 1} \subset L^\infty(\tilde \Omega \times [0,T];H^{1,2})$ and thereby prove claim (vii), using the same upgrade to $L^2([0,T];H^{1,2})$ convergence as in the proof of Prop. \ref{WeakApproxSystemExistence}.

\textbf{Claim 3:} With this at hand, we prove the following facts:
\begin{enumerate}[(a)] 
    \item \label{item:MartingaleIdentification} $M$ is an ${L}^2$-valued continuous time martingale on $[0,T]$, adapted to the filtration $\mathcal F_t = \sigma(\phi|_{[0,s]}, \bm c|_{[0,s]},\bm c_s\,;0<s \leq t)$, with covariance operator 
    $$\int_0^t  b(\phi_s,\bm c_s) Q  b^\ast(\phi_s,\bm c_s) \d s.$$
    \item \label{item:PositiveTimeWeakForm} For any $t>s > 0$ and $\bm v \in {H}^{1,2}_b$, 
    $$\begin{aligned}
    \langle \bm c_t, \bm v \rangle &= \langle \bm c_s, \bm v \rangle + \int_s^t - \langle \bm D \nabla \bm c_r, \nabla \bm v \rangle  + \langle \bm D\frac{\nabla \bm c_r\nabla \phi_r}{\phi_r}, \bm v\rangle  + \langle  f(\phi_r,\bm c_r),\bm v \rangle  \d r \\ &\quad+ \int_s^t \langle \bm v, \d M_r. \rangle  
    \end{aligned}$$ In  particular, $c$ is weakly continuous on $(0,T]$ as an ${L}^2$-valued process. 
    \item \label{item:PhiWeightedWeakForm}For all $\alpha = \frac1k$,  $k \in \mathbb N$, $t > 0$  and $\bm v \in {H}^{1,2}_b$, \begin{equation}\label{Gamma}
    \begin{aligned}
    \langle \phi_t^\alpha\bm c_t, \bm v \rangle =  \langle \phi_0^\alpha\bm c_0, \bm v \rangle &+ 
    \int_0^t\Big(- \langle \bm D \phi_s^{\alpha}\nabla \bm c_s, \nabla \bm  v  \rangle - \langle ((\bm D+ \gamma \bm I)\alpha-\bm D) \phi_s^{\alpha-1}   \nabla \bm c_s \nabla \phi_s, \bm v \rangle\\  &  \quad\quad +\langle  \phi_s^{\alpha} f(\phi_s, \bm c_s), \bm v \rangle  - \alpha \gamma \langle \phi_s^{\alpha-1}  \bm c_s \otimes \nabla \phi_s , \nabla \bm v \rangle
    \\ & \quad\quad - \alpha(\alpha-1)\gamma \left \langle \frac{|\nabla \phi_s|^2}{\phi^{2-\alpha}_s} \bm c_s, \bm v\right \rangle + \alpha \langle \phi_s^{\alpha-1} \Psi(\phi_s, \bm c_s,\nabla \phi_s) \bm c_s, \bm v \rangle \\ & \quad\quad + \alpha \langle \phi_s^{\alpha-1} g(\phi_s, \bm c_s) \bm c_s , \bm v \rangle\Big) \d s 
    \\ &  + \int_0^t \langle \bm v, \phi_s^{\alpha} \d M_s \rangle .
    \end{aligned}
    \end{equation}
    and hence $\phi^\alpha \bm c$ is weakly continuous in ${L}^2$.
\end{enumerate}
We begin begin by proving \eqref{item:PositiveTimeWeakForm}. For each $k$, this works by analogous (but simpler) methods as the proof of Prop. \ref{WeakApproxSystemExistence}, where in particular, for each $\omega' \in \Omega'$, we get a uniform (in $\Omega' \times [T/k,T]$) bound on $\frac{\nabla \phi^{\epsilon_n}}{\phi^{\epsilon_n}}$ dictated by $\int \phi_0(\omega)\d \bm x$. Moreover, 
\begin{equation} \label{trivial}
\int_0^t \langle \bm v, \cdot \rangle  \d M^{\epsilon_n}_s = \langle \bm v, M^{\epsilon_n}_t \rangle \overset{n \to \infty}{\to} \langle \bm v, M_t\rangle
\end{equation} 
is trivial by the convergence properties of $M^{\epsilon_n}$.

As in Prop. \ref{StationaryApproximation}, \eqref{item:MartingaleIdentification} follows by almost sure uniform convergence of $M^{\epsilon_n}$ and uniform integrability in the dual space together with almost sure (weak) convergences of $(\phi^{\epsilon_n},\bm c^{\epsilon_n}, \bm c^{\epsilon_n}_t)_{n \geq 1}$. We note that in this case, claim \eqref{item:PositiveTimeWeakForm} implies $\sigma(\phi|_{[0,t]}, \bm c|_{[0,t]}, \bm c_s, \bm c_t)\subset \mathcal F_t$-measurability of $M_t - M_s$ for any $s > 0$, and taking a limit $s \to 0$ yields the required measurability of $M_t$. 

At last, we can apply the product rule from Lemma \ref{MyProdIto} to $\phi^\alpha_t \bm c_t - \phi^\alpha_s\bm c_s$ and combine the continuity properties of $\phi^\alpha \bm c$ with the integrability properties of $\frac{\nabla \phi}{\phi^{1-\alpha/2}}$ and $\phi^\alpha \nabla \bm c$ to take the limit $s \to 0$. This then gives the weighted formula \eqref{item:PhiWeightedWeakForm}.

\textbf{Claim 4:} The processes $(\phi, \bm c)$ constitute a weighted martingale solution of equation \eqref{eq:WeakSolFullSystem}. 

To this end, we take the limit $\alpha \to 0$ in \eqref{Gamma}. Let $\rho \in L^2(\tilde \Omega;L^2([0,T]; H^{1,2})) \cap L^\infty(\tilde \Omega;L^\infty([0,T];L^\infty))$ with $$\tilde{\mathbb E}\left[\int_0^T \int \frac{|\nabla \phi_t|^2}{\phi^2_t}\rho^2_t \d {\bm x} \d t\right]< \infty$$ be given. Further suppose that $\rho \in L^\infty(\tilde \Omega \times [0,T]\times \mathbb T^d)$ has version that is progressively measurable with respect to the filtration generated by $(\phi, \bm c)$ as processes with values in $L^2 \oplus {L}^2$ and that $\rho$ is absolutely continuous in $(H^{1,2}_b)$ with $$\partial_s \rho \in L^2(\tilde \Omega;L^2([0,T];(H^{1,2})^\ast) + L^1(\tilde \Omega;L^1([0,T];L^1)).$$

Application of Lemma \ref{MyProdIto} gives us
$$
\begin{aligned}
    \langle \phi_t^\alpha \rho^2_t\bm c_t, \bm v \rangle =  \langle \phi_0^\alpha \rho^2_0\bm c_0, \bm v \rangle &+ 
    \int_0^t\Big(- \langle \bm D \phi_s^{\alpha} \rho^2_s\nabla \bm c_s, \nabla \bm v  \rangle - \langle \bm D \phi_s^{\alpha} \nabla \bm c_s, 2\rho_s\nabla \rho_s \otimes \bm v \rangle \\ & \quad  \quad  - \langle ((\bm D+ \gamma \bm I)\alpha-\bm D) \phi_s^{\alpha-1} \nabla \bm c_s \nabla \phi_s, \rho^2_s \bm v \rangle\\  &  \quad\quad +\langle  \rho^2_s\phi_s^{\alpha} f(\phi_s, \bm c_s), \bm v \rangle  - \alpha \gamma \langle \phi_s^{\alpha-1}  \bm c_s \otimes \nabla \phi_s , \rho^2_s \nabla \bm v + 2\rho_s \bm v \otimes \nabla \rho_s\rangle
    \\ & \quad\quad - \alpha(\alpha-1)\gamma \left \langle \frac{|\nabla \phi_s|^2}{\phi^{2-\alpha}_s} \bm c_s, \rho^2_s\bm v\right \rangle + \alpha \langle \phi_s^{\alpha-1} g(\phi_s) \bm c_s , \rho^2_s \bm v \rangle \\ & \quad\quad + \alpha \langle \phi_s^{\alpha-1} \Psi(\phi_s, \bm c_s,\nabla \phi_s) \bm c_s, \rho^2_s \bm v \rangle \\ &\quad \quad \quad + 2\langle \partial_s \rho_s \bm v, \rho_s \phi^\alpha_s \bm c_s \rangle \Big) \d t 
    \\ &  + \int_0^t \langle \rho^2_s \bm v, \phi_s^{\alpha} \rangle \d M_s .
\end{aligned} $$
If we can now prove that $(\phi^\alpha \rho \nabla c)_{\alpha > 0}$ is uniformly bounded in $L^2(\tilde \Omega\times[0,T];{L}^2)$, the monotone convergence theorem gives almost sure boundedness of $\norm{\rho_t \nabla c_t}^2_{ L^2([0,T];{L}^2)}$. Then, the fact that $\alpha \frac{\nabla \phi}{\phi^{1-\alpha}} \to 0$ in $L^2([0,T];{L}^2)$ and the Lebesgue DCT imply that for $\alpha \to 0$,
$$
\begin{aligned}
    \langle \rho^2_t\bm c_t, \bm v \rangle =  \langle \mathds{1}_{\{\phi_0 > 0\}}\rho^2_0\bm c_0, \bm v \rangle &+ 
    \int_0^t\Big(- \langle \bm D \nabla \bm c_s, \nabla (\rho^2_s\bm v)  \rangle   + \langle \bm D \phi_s^{-1} \nabla \bm c_s  \nabla \phi_s, \rho^2_s \bm v \rangle\\  &  \quad\quad +\langle  \rho^2_s f(\phi_s, \bm c_s), \bm v \rangle   + 2\langle \partial_s \rho_s\bm v, \rho_s \bm c_s  \rangle \Big) \d t 
    \\ &  + \int_0^t \langle \rho^2_s \bm v, \cdot \rangle \d M_s.
\end{aligned} $$
We first explain individual convergences of the involved terms.
\begin{enumerate}[(1)]
    \item Note that $ \mathds{1}_{\{\phi_0 > 0\}}\rho^2_0 \equiv \rho^2_0$ by Proposition \ref{prop:WeightSupport}. 
    \item Next, we prove the convergence 
    $\langle \partial_s \rho_s \bm v, \rho_s \phi^\alpha_s \bm c_s\rangle \overset{n \to \infty}{\to} \langle \partial_s \rho_s \bm v, \rho_s \bm c_s \rangle.
    $ This requires two separate arguments: An estimate on the $L^1([0,T];L^1)$-part of $\partial_s\rho_s$, and an estimate on the $L^2([0,T];(H^{1,2})^\ast)$-part.
    \item The convergence of the $L^1$-part is a consequence of the Lebesgue DCT since $\phi_s$ is strictly positive for $s > 0$, which implies almost sure convergence, and the existence of the majorant given by $\rho \bm c_s \bm v$. 
    \item The other convergence follows from $\rho_s \phi^\alpha_s \bm c_s \odot \bm v \to \rho_s \bm c_s \odot \bm v$ in $L^2(\tilde \Omega \times [0,T]; {H}^{1,2})$, in particular $$\nabla (\rho_s \phi^\alpha_s \bm c_s \odot \bm v) = \phi^\alpha_s \nabla(\rho_s\bm c_s \odot\bm v) + \alpha \frac{\nabla \phi_s}{\phi^{1-\alpha}_s} \otimes \rho_s\bm c_s \odot\bm v \to \nabla(\rho_s\bm c_s \odot \bm v)$$ in $L^2(\tilde \Omega \times [0,T] \times \mathbb T^d)$. It is not difficult to derive by the DCT that the first term converges to the right hand side, while $L^2$-convergence of $\nabla(\phi^\alpha) = \alpha \frac{\nabla \phi_s}{\phi^{1-\alpha}_s}\to 0$ settles the proof.
    \item By similar arguments, all terms with a prefactor $\alpha$ vanish, where we use in particular admissibility of the weight.
\end{enumerate}
It is left to show that $(\phi^\alpha \rho \nabla \bm c)_{\alpha > 0}$ is bounded uniformly, the proof of which requires the energy inequality given by Lemma \ref{MyProdIto}. As in the proof of Lemma \ref{LastApriori}, we can rearrange the Itô formula for $\norm{\phi^{\alpha/2}_t \rho_t \bm c_t}^2_{L^2}$ to then obtain the desired boundedness. The proof is mostly analogous. A suitable control of the additional term deriving from $\partial_t \rho_t$ can be derived since the Young inequality gives us a constant $C_\varepsilon > 0$ with
$$\begin{aligned}
\int_0^t\langle \partial_s \rho_s, \rho_s \phi^{\alpha}_s \underbrace{\bm c^2_s}_{\coloneqq \bm c_s \odot \bm c_s}\rangle \d s &\leq C_\varepsilon\int_0^t\norm{\partial_s \rho_s}^2_{({H}^{1,2}_b)^\ast} \d s + \varepsilon \int_0^t\norm{ \rho_s \phi^{\alpha}_s \bm c^2_s}^2_{{H}^{1,2}_b} \d s \\
& \leq C_\varepsilon\int_0^t\norm{\partial_s \rho_s}^2_{({H}^{1,2}_b)^\ast} \d s + \varepsilon \int_0^t \norm{\rho_s \phi^\alpha_s \bm c^2_s}^2_{{L}^2} \d s+ \const\\ & \quad + 3\varepsilon \int_0^t \norm{\nabla \rho_s \otimes \phi^\alpha_s \bm c^2_s}^2_{{L}^2} + \norm{ \rho_s \nabla (\phi^\alpha_s) \otimes \bm c^2_s}^2_{{L}^2} + \norm{2\rho_s \phi^\alpha_s \bm c_s \odot \nabla \bm c_s}^2_{{L}^2} \d s
\end{aligned}$$
for any $\varepsilon > 0$. Here, we in particular used $L^\infty(\Omega \times[0,T]\times \mathbb T^d)$-boundedness of $\rho, \phi$ and $\bm c$. In particular, for $\varepsilon > 0$ small enough, $\norm{\rho_s \phi^\alpha_s \bm c_s  \odot\nabla \bm c_s}^2_{{L}^2}$ gets absorbed by $-\norm{\sqrt{\bm D} \rho_s \nabla \bm c_s}^2_{{L}^2}$.
Altogether, we can conclude that the desired type of martingale solution exists. The additional continuity and integrability properties of $\rho^2\bm c$ follow by integrability of $\rho \nabla \bm c \in L^2(\Omega \times [0,T] \times \mathbb T^d)$ and weak continuity of $\rho^2 \bm c$, combined with continuity of $t \mapsto \norm{\rho^2_t \bm c_t}^2$ due to Lemma \ref{MySquareIto}.
\end{proof}

\section{Applications} \label{ApplicationSec}
In this section, we present applications of the theory developed in this manuscript. The focus is on models of chemotaxis from biophysical publications.

\begin{example} In \cite{ABS}, the authors first introduced a system of equations directly corresponding to the one we studied in the previous sections. The model in question specifies two processes $\phi, c \colon [0,T] \times \mathbb T^2 \rightarrow \mathbb R$ intended to describe the motion of a cell coupled to the dynamics of motion-inducing biochemical components inside the cell. It consists of the random reaction-diffusion equation \begin{equation}\label{PhaseFieldABS} 
    \partial_t \phi_t  = \gamma \Delta \phi_t + g(\phi_t) + \beta \left(\int\phi_t \d x - A_0 \right)|\nabla \phi_t| + \alpha c_t |\nabla \phi_t|,
\end{equation} which models time evolution of the phase field $\phi$, coupled with the stochastic reaction-diffusion equation \begin{equation} \label{ChemEqABS}
\partial_t c_t =  \frac{1}{\phi_t} \nabla \cdot \big(\phi_t D \nabla c_t\big) + f(c_t) - \rho c_t +  \phi_t(1-\phi_t) \xi_t,
\end{equation} to model the distribution $c$  inside the cell. Here $\gamma, \beta, A_0, \alpha, D, \rho >0$ are real constants. The reaction terms $g$ and $f$ are given by \[g(x) = -Kx(x-1)(x-0.5)\] and \[f(x) = -K_\alpha x(x-1)(x-\delta(\phi_t,c_t)),\] with \[\delta(\phi_t,c_t) = \delta_0 + M\left( \int \phi_t c_t \d x-A_1 \right),\] where $K, K_\alpha, M, A_1 \in \R^+$ and $\delta_0 \in (0,1)$. The spatio-temporal noise $\xi$ specified in the article is of Ornstein-Uhlenbeck type. 

As Ornstein-Uhlenbeck type noise does not satisfy the martingale properties we relied on in the derivations of our existence theorem, we assume that the noise is of Wiener type. To fit this system into the framework of this article, we need to replace the nonlocal reaction threshold $\delta$ by the truncated term $\tilde \delta(\phi_t, c_t) = \delta(\tilde \phi_t, \tilde c_t)$ (cf. \eqref{eq:WeakSolTruncFullSystem} with $K_\phi = K_1 = 1$, $L_1 = 0$). Similarly, we truncate the nonlocal factor in \eqref{ChemEqABS}. At last, we incorporate the truncation $\eta$ of the dispersion coefficient.

By the invariance properties of the resulting reaction terms, we can conclude from Theorem \ref{MartingaleSolutionLog} that variational solutions of the system $$\begin{aligned}
\partial_t \phi_t  &= \gamma \Delta \phi_t + g(\phi_t) + \beta \left(\int\phi_t \d x - A_0 \right)|\nabla \phi_t| + \alpha c_t |\nabla \phi_t| \\
\d c_t &= \left( \frac{1}{\phi_t} \nabla \cdot \big(\phi_t D \nabla c_t\big) + f(c_t) - \rho c_t \right) \d t+  \eta(c_t)\phi_t(1-\phi_t) \d W_t,
\end{aligned}$$
exist, given initial conditions $0 \leq \phi_0, c_0 \leq 1$ with $\phi_0,\in  H^{1,2}$ and $\log \phi_0 \in L^1$. Here, $W_t$ is assumed to be a $Q$-Wiener process on $H^{r,2}(\mathbb T^2)$ for some $r > 0$.
\end{example}
\begin{example}
Related (deterministic) models have also been applied to population dynamics in changing environments, though not in the framework of phase-field models. \citet{PeaseModel} introduced the system of equations 
$$ \begin{aligned} \partial_t n_t &= \frac{\sigma^2}{2}\Delta n_t + n_t\log W_t \\ \partial_t z_t &= \frac{\sigma^2}{2} \Delta z_t + \sigma^2 \nabla \log n_t \nabla z_t + G\,\partial_z \log W_t \end{aligned}$$
to describe the evolution of the density $n$ of individuals of a population, coupled to the evolution of the mean phenotype $z$. Here, $\sigma, G > 0$ are constants and $W$ is the per-capita growth rate of a population at a particular point in space. To the best of our knowledge, though frequently used in the biological literature (cf. \cite{KirkpatrickModel} and subsequent works such as \cite{GarciaRamosModel}), 
the equations have been analysed so far only in special cases. Examples are \cite{MillerWaveKirkpatrick} and  \cite{RaoulKirkpatrick}, where the latter also considers this type of model under periodic boundary conditions. Reference  \cite{KanarekWebb} considers the case $\log W_t = (1-n_t)(n_t-z^2_t)$ 
which is covered by our assumptions. Setting $\sigma^2 = 2, G = 1$ and additionally introducing noise to the phenotype evolution results in the system $$ \begin{aligned} \partial_t n_t &= \Delta n_t + n_t(1-n_t)(n_t-z_t^2) \\ \d z_t &= \left( \Delta z_t + 2 \frac{\nabla n_t \nabla z_t}{n_t} - 2G(1-n_t)z_t\right) \d t + \eta(z_t) \d W_t.\end{aligned}$$ Again, the truncation $\eta$ is supported inside the interval $[0,1]$ and the driving noise is chosen as in the previous example. These considerations extend to the Kirkpatrick-Barton model considered in \cite{RaoulKirkpatrick}, for appropriate choices of the parameter $y_{opt}$. Note that although the constant factor of the logarithmic derivative differs from the coefficient of the Laplacian, the derived solution theory applies with virtually no difference. 

We note that the SPDE model introduced in \cite{BartonSPDE} is currently out of reach of the theory developed in this article, due to the singular dispersion coefficients and the presence of noise in the entire system. However, if we discard stochastic effects on the population density $n$, replace white noise by coloured noise, set $b \equiv 0$ and regularise the singular dispersion coefficients as $\sqrt{\frac{z(1-z)}{n+\epsilon}}$ for some $\epsilon > 0$, our analysis is applicable.

\end{example}

\begin{example} In \cite{CaoRappelELife, CaoRappelNr2}, the authors introduced stochastic phase-field models of the form
$$
\begin{aligned}
\partial_t \phi_t &= \gamma \Delta \phi_t +  g(\phi_t)+ \alpha \frac{S^3_t}{S^3_t + S_2^3}|\nabla \phi_t| 
- \beta \left( \int \phi_t \d x - A_0 \right)|\nabla \phi_t|\\
\partial_t (\phi_t R_t) &= \nabla \cdot (\phi_t  D_R\nabla R_t) + \phi_t \left[ \frac{c_2 S_t - c_1 R_t}{\tau} + \xi^1_t \right]\\
\partial_t (\phi_t S_t) &= \nabla \cdot (\phi_t  D_S\nabla S_t) + \phi_t \left[ \left( \frac{k_s S^2_t}{K_s^2 + S^2_t} + b \right)(S_1 - S_t) - (d_1 + d_2 R_t)S_t + \xi^2_t \right],
\end{aligned}
$$
for $\alpha, \gamma, \beta, A_0$ and $g$ chosen as in the previous example, noise terms $\xi^1, \xi^2$ and constants $$S_1,S_2,k_s,K_s, c_2, c_1, d_1,d_2,D_R, D_S > 0$$
Their numerical methods (cf. the Materials and methods section) indicate that the appropriate mathematical formulations of the stochastic partial differential equations are
$$\d R_t = \left(\frac{1}{\phi_t}\nabla \cdot (\phi_t  D_R\nabla R_t) + \frac{c_2 S_t - c_1 R_t}{\tau} - \partial_t \log \phi_t \cdot R_t \right) \d t + \d W^1_t$$
and
$$
\d S_t =\left( \frac{1}{\phi_t} \nabla \cdot (\phi_t D_S\nabla S_t) + \left( \frac{k_s S^2}{K_s^2 + S^2_t} + b \right)(S_1 - S_t) - (d_1 + d_2 R_t)S_t - \partial_t \log \phi_t \cdot S_t  \right)\d t + \d W^2_t.
$$
As the methods developed in this manuscript do not extend to this model due to the presence of $\partial_t \log \phi_t$, we need to introduce the simplifying assumption that the phase-field dynamics happen at a slower time-scale to approximate $$\d (\phi_t R_t) \approx \phi_t \d R_t.$$ Also, once again, we truncate the noise with some cut-off function $\eta$, for example inside the interval $[0,S_1]$. The resulting system $$\begin{aligned}
\d R_t &= \left(\frac{1}{\phi_t}\nabla \cdot (\phi_t  D_R\nabla R_t) + \frac{c_2 S_t - c_1 R_t}{\tau} \right) \d t + \eta(R_T)\d W^1_t\\
\d S_t &=\left( \frac{1}{\phi_t} \nabla \cdot (\phi_t D_S\nabla S_t) + \left( \frac{k_s S^2_t}{K_s^2 + S^2_t} + b \right)(S_1 - S_t) - (d_1 + d_2 R_t)S_t \right)\d t + \eta(S_t)\d W^2_t.
\end{aligned}
$$
can be treated by the methods we developed. In particular, we note that for $R_0, S_0 \in [0,S_1]$, the reaction terms force the system to stay confined to this interval, i.e. Assumption \ref{ass:WeakSolInvariance} is satisfied. We remark that analogous considerations can be applied to the models employed in \cite{FlemmingCorticalWaves} and \cite{MoldenhawerSpontaneous}, where the latter results in exactly the system of equations considered in the previous example.

\end{example}

\begin{example}
In \cite{Torres2025dissipative}, a conservative reaction-diffusion system is employed to model cell motility states. For $D, D_F, b, \gamma, s, \eta, p_0, p_1 > 0$, consider
$$\begin{aligned}
\partial_t u_t &= D \Delta u_t + (b + \gamma u_t^2)v_t - (1 + sF_t + u^2_t)u_t , \\
\partial_t v_t &= \Delta v_t -(b + \gamma u_t^2)v_t + (1 + sF_t + u^2_t)u_t, \\
\partial_t F &= D_F \Delta F_t + \eta(p_0 + p_1 u_t - F_t).
\end{aligned}$$ Although this deterministic model does not involve a phase-field, our solution theory  even in the stochastic, coupled case still applies. W.l.o.g. set $b = \gamma = \frac12$, $D = D_F = s = \eta = 1$, $p_0 + p_1 = 1$. Then the specified reaction term satisfies the invariance condition from Assumption \ref{ass:WeakSolInvariance} for $\mathcal K = [0,1]^3$. Thus, given initial conditions confined to $\mathcal K$, suitably smooth $Q$-Wiener processes $W^1, W^2, W^3$ and a smooth truncation $\eta$ with support on $[0,1]$, there exist solutions on the two-torus $\mathbb T^2$ of the system 
$$\begin{aligned}
\partial_t \phi_t &= \gamma \Delta \phi_t +  g(\phi_t)+ \alpha h(u_t,v_t,F_t)|\nabla \phi_t| 
- \beta \left( \int \phi_t \d x - A_0 \right)|\nabla \phi_t|\\
\d u_t &= \left(\frac{1}{\phi_t}(\phi_t D \Delta u_t) + (b + \gamma u_t^2)v_t - (1 + sF_t + u^2_t)u_t \right)\d t+\eta(u_t)\d W^1_t, \\
\d v_t &= \left(\frac{1}{\phi_t}(\phi_t \Delta v_t) -(b + \gamma u_t^2)v_t + (1 + sF_t + u^2_t)u_t\right)\d t+\eta(v_t)\d W^2_t, \\
\d F &= \left( \frac{1}{\phi_t}(\phi_t D_F \Delta F_t) + \eta(p_0 + p_1 u_t - F_t)\right)\d t + \eta(F_t)\d W^3_t.
\end{aligned}$$ Here, the parameters that specify the phase-field dynamics are chosen as in the previous examples, except for the function $h$, which can be chosen to be any smooth function.
Naturally, this existence result extends to a large class of other conservative systems which satisfy a suitable invariance condition.
\end{example}

\section{Appendix}
\begin{proof}[Proof of Lemma \ref{MyProdIto}]
By the specified conditions, we can infer that $x$ and $y$ have  weakly continuous representants in $L^2$ with 
$$\sup_{t \in [0,T]} \norm{ x_s}_{L^\infty}, \sup_{t \in [0,T]} \norm{y_s}_{L^\infty} \in L^\infty(\Omega).$$ Here, the supremum runs over all $t \in [0,T]$, which is well-defined since balls in $L^\infty$ are strongly closed and convex, hence closed under weak convergence.
Choose a sequence of partitions $\pi_n$ such that the sampled step functions $$x^n = \sum_{t_i \in \pi_n} x_{t_{i+1}}\mathds{1}_{[t_i,t_{i+1})}, ~ y^n = \sum_{t_i \in \pi_n} y_{t_{i}}\mathds{1}_{[t_i,t_{i+1})} \in H^{1,2}$$ of the continuous representant approximate $x_s$ and $y_s$ in $L^2(\Omega\times[0,T];H^{1,2})$ (cf. Lemma 4.2.6 in \cite{LiuRoeckner2015}). As usual, we can extract 
a subsequence, for simplicity again denoted by $n \geq 1$, such that $x_n, y_n$ converge $\mathbb P \otimes \d t \otimes \d x$-almost surely. A telescopic sum now yields 
$$ 
\begin{aligned}
    \langle x_s y_s, v \rangle - \langle x_0y_0, v\rangle &= \sum_{t_i \in \pi_{n_k}} \langle x_{t_{i+1}}-x_{t_{i}}, y_{t_i}v\rangle+\langle y_{t_{i+1}}-y_{t_{i}}, x_{t_{i+1}}v\rangle \\ &=\int_0^t \langle u_s+ \tilde u_s, y^n_s w\rangle \d s+\int_0^t \langle v_s+ \tilde  v_s, x^n_sw\rangle \d s+ \int_0^t\langle y^n_sw, \cdot \rangle \d M_s.
\end{aligned}
$$
Note that by construction, $y^n_s$ is adapted. The approximation properties of $x^n, y^n$ yield 
convergence of the respective terms to the desired limit. In particular, one can exploit uniform 
$L^\infty$-boundedness and almost sure convergence to apply the Lebesgue DCT and conclude that 
$$
\int_0^t \langle \tilde u_s, y^n_s w\rangle + \langle \tilde v_s, x^n_s w\rangle \d s 
\to \int_0^t \langle \tilde u_s, y_s w\rangle + \langle \tilde v_s, x_s w\rangle \d s.
$$ 
At last, we show almost sure convergence of the stochastic integral (for a suitable subsequence). We 
do this by showing convergence in $L^2(\Omega;C([0,T];L^2))$. As a preliminary step, we remark that the 
almost sure convergence of $$\int_0^t \norm{(y_s-y^n_s)w g_s \sqrt Q}^2_{HS} \d s = \sum_{k \geq 1} \int_0^t \norm{(y_s-y^n_s)w g_s \sqrt Q u_i}^2_{L^2} \d s\to 0$$ follows by the 
Lebesgue DCT and almost sure convergence of 
$$
\int_0^t \norm{(y_s-y^n_s)w g_s \sqrt Q u_i}^2_{L^2} \d s 
= \int_{[0,t] \times \mathbb T^n} ((y_s-y^n_s)w g_s \sqrt Q u_i)^2\d x \d s \to 0 
$$ 
due to the pointwise upper bound $$\norm{(y_s-y^n_s)w g_s \sqrt Q u_i}^2_{L^2} \leq 2\norm{y_s}_{L^\infty }\norm{w}_{L^\infty} \norm{g_s \sqrt Q u_i}^2_{L^2}$$ for each $i \geq 1$. The 
claim now follows by means of the BDG inequality and another application of the DCT.
\end{proof}

\begin{proof}[Proof of Lemma \ref{MySquareIto}]
Let $P_n$ denote the projection onto an orthogonal set of vectors $$A_n = \Set{e^n_1,\dots,e^n_{k_n}} \subset H^{1,2} \cap L^\infty$$ that possesses the $H^{1,2}$ and $L^\infty$ stability property, so that $$\norm{P_n x}_{H^{1,2} \cap L^\infty} \leq C_P \norm{x}_{H^{1,2} \cap L^\infty}$$ for some $C_P > 0$ independent of $n$, and that $\norm{P_n x - x}_{L^2} \to 0$. We first prove that for any finite projection $P_n$, \begin{equation}
\begin{aligned} \label{finiteProjIto}
\norm{P_n x_t}^2_{L^2} &= \norm{P_nx_0}^2_{L^2} + \int_0^t \langle u_s, P_n x_s\rangle \d s + \int_0^t \langle \tilde{u}_s, P_n x_s \rangle \d s + \int_0^t \langle P_n x_s, \cdot \rangle \d M_s \\&\quad +\int_0^t \norm{P_n g_s \sqrt Q}^2_{HS} \d s.
\end{aligned}
\end{equation}
Due to polarisation, $$\norm{P_nx_t}^2_{L^2} - \norm{P_nx_0}^2_{L^2} = \norm{P_n(x_t-x_0)}^2_{L^2} - 2\langle x_t-x_0, P_n x_0\rangle.$$ Now, by the properties of the chosen basis, $P_n x_0 \in H^{1,2} \cap L^\infty$ and it follows that \begin{equation} \label{term2}
\langle x_t-x_0, P_n x_0\rangle = \int_0^t \langle u_s, P_n x_0\rangle \d s + \int_0^ t \langle \tilde{u}_s, P_n x_0\rangle \d s + \int_0^t \langle P_n x_0, \cdot \rangle \d M_s.
\end{equation} On the other hand, $$\norm{P_n(x_t-x_0)}^2_{L^2}  = \sum_{k = 1}^n \langle x_t-x_0, e_k \rangle^2,$$ where, due to the Ito formula applied to  $$f_t =\langle x_t-x_0, e_i \rangle = \int_0^t \langle u_s, e_k\rangle + \langle \tilde{u}_s, e_k \rangle \d s + \int_0^t  \d \langle e_k, M_s\rangle ,$$ we find that \begin{equation*} 
\begin{aligned}
\langle x_t-x_0, e_k \rangle^2 = &2 \int_0^t \langle x_s-x_0,e_k\rangle (\langle u_s, e_k\rangle + \langle \tilde{u}_s, e_k\rangle) \d s \\ &+ 2\int_0^t \langle x_s -x_0,e_k\rangle \d \langle e_k,M_s\rangle + \int_0^t \norm{\langle e_k, g_s \sqrt{Q} \circ \cdot \rangle}^2_{HS} \d s.
\end{aligned}
\end{equation*}
Plugging in and collecting terms, we therefore arrive at the identity
\begin{equation} \label{term1}
\begin{aligned}
\norm{P_n(x_t-x_0)}^2_{L^2}  &= 2\int_0^t \langle u_s, P_n(x_s-x_o) \rangle + \langle \tilde{u}_s,P_n(x_s-x_0)\rangle \d s \\ &\quad + 2 \int_0^t \langle P_n (x_s-x_0),\cdot \rangle \d M_s + \int_0^t \norm{P_n g_s \sqrt{Q}}^2_{HS} \d s.
\end{aligned}
\end{equation}
Adding \eqref{term2} and \eqref{term1} together, we obtain \eqref{finiteProjIto}. At last, we take the limit $n \to \infty$ to derive the desired identity. Naturally, $\norm{P_n x_t}^2_{L^2}$ and $\norm{P_n x_0}^2_{L^2}$ converge to the appropriate limit. Further, note that $$\int_0^t \norm{P_n x_s -x_s}^2_{L^2} \d s \to 0$$ by dominated convergence, and therefore $P_n x \to x$ in $L^2([0,T];L^2)$. From this, it already follows that the stochastic integral and trace term converge, again by dominated convergence. Further, we find that there exists a subsequence $(n_k)_{k \geq 1}$ such that $P_n x \to x$ almost surely on $[0,T] \times \mathbb T^n$. It then finally follows that $$\int_{[0,T] \times \mathbb T^n} \tilde{u}_s(y) P_n x_s(y) \d y \to \int_{[0,T] \times \mathbb T^n} \tilde{u}_s(y)x_s(y) \d y$$ by dominated convergence, since $$\norm{P_n x_s} _{L^\infty([0,T] \times \mathbb T^n)} \leq C_P\norm{x_s}_{L^\infty([0,T] \times \mathbb T^n)}$$ by the stability properties of the projections. Finally, weak convergence in $L^2([0,T];L^2)$ of $\nabla P_n x_s$ to $x_s$ follows by boundedness, which ensures convergence of some subsequence, and strong convergence of $P_n$ in $L^2$, so that the argument is finished by density of smooth enough functions.
\end{proof}
\begin{proof}[Proof of Proposition \ref{prop:WeightSupport}]
We assume that for all $t \in [0,T]$, $\phi_t \in \mathcal X_\phi$ and that $\rho \in L^\infty([0,T] \times \mathbb T^n)$. By the chain rule,
\begin{equation} \label{eq:WeightSupport}
\begin{aligned}
&\int (1-\phi^\alpha_t)\rho^2_t\d x - \int (1-\phi^\alpha_0)\rho^2_0 \d x \\&= -\alpha\left((1-\alpha)\int_0^t \int \frac{|\nabla \phi_s|^2}{\phi^{2-\alpha}_s} \rho^2_s\d x \d s + \int_0^t \int \tilde \Psi(\bm c_t,\phi_t) \frac{|\nabla \phi_t|}{\phi^{1-\alpha}_t} \rho^2_s\d x + \int_0^t \int \frac{\Tilde{g}(\phi_t)}{\phi_t} \phi^\alpha_s \rho^2_s\d x \d s  \right) \\ &\quad + 2 \int_0^t \langle \partial_s\rho_s, \rho_s(1-\phi^\alpha_s)\rangle \d s.
\end{aligned}
\end{equation}
We now want to let $\alpha \to 0$. By assumption, the first term on the right hand side tends to $0$ as $\alpha \to 0$. Similarly, $$\rho_s (1-\phi^\alpha_s) \to 0 \text{ in } L^2([0,T];L^2), ~\nabla(\rho_s(1-\phi^\alpha_s)) = \nabla \rho_s (1-\phi^\alpha_s)-\alpha \rho_s \frac{\nabla \phi_s}{\phi^{1-\alpha}_s} \to 0 \text{ in } L^2([0,T];L^2).$$ Although we cannot additionally conclude that $$\rho_s(1-\phi^\alpha_s) \overset{L^\infty([0,T];L^\infty)}{\to} 0,$$ dominated convergence still gives us that $$\int_0^t \langle \partial_s\rho_s, \rho_s(1-\phi^\alpha_s)\rangle \d s \to 0.$$ Taking the limit $\alpha \to 0$ on the left hand side then gives that $\int \mathds{1}_{\{\phi_0=0\}}\rho^2_0 \d x = 0$, since $\phi_t$ is strictly positive. Therefore, $\mathds{1}_{\{\phi_0=0\}} \rho_0 = 0$.
\end{proof}

\begin{proof}[Proof of Example \ref{HeatEqWeight}]
Adaptedness and boundedness properties are immediate by the deterministic properties of the heat equation. Also, by analogous methods as in the proof of Lemma \ref{arbitraryphibound}, we find that $$\int_0^t \norm{\frac{\nabla h_s}{\sqrt{h_s}}}^2_{L^2} \d s < \infty.$$ Using the product rule, we now derive that $$\begin{aligned}
&\int h_t \log(\phi_t+\epsilon) \d x -\int  \phi_0\log(\phi_0+\epsilon)\d x \\ &= \int_0^t- 2\int \frac{\nabla h_s \nabla \phi_s}{\phi_s+\epsilon} \d x +\int h_s \frac{|\nabla \phi_s|^2}{\phi^2_s} \d x + \int  h_s \frac{\Psi(\phi_s, \bm c_s, \nabla \phi_s)}{(\phi_s+\epsilon)} \d x + \int h_s \frac{g(\phi_s, \bm c_s)}{(\phi_s+\epsilon)} \d x \d s  
\end{aligned}
$$
Again, we can rearrange and apply Young's inequality to control the third and fourth term on the right hand side. Combined with the estimate $$\int \frac{\nabla h_s \nabla \phi_s}{\phi_s+\epsilon} \d x \leq \const \int \frac{|\nabla h_s|^2}{h_s} \d x + \delta \int h_s \frac{|\nabla \phi_s|^2}{(\phi_s+\epsilon)^2} \d x$$ for $\delta \ll 1$, we can rearrange the product identity and take a limit $\epsilon \to 0$ to conclude existence of a deterministic constant with  
\begin{equation*}
\int_0^t  \int h_s \frac{|\nabla \phi_s|^2}{\phi_s^2} \d x  \d s < \infty. \qedhere
\end{equation*}
\end{proof}

\section*{Acknowledgement}
AS is supported by Deutsche Forschungsgemeinschaft (DFG, German Research
Foundation) under Germany's Excellence Strategy - The Berlin Mathematics
Research Center MATH+ (EXC-2046/1, project ID: 390685689). WS acknowledges support 
for DFG CRC/TRR 388 “Rough Analysis, Stochastic Dynamics and Related Topics”, Projects A10, B09.

\bibliography{maths}
\end{document}